\documentclass[11pt]{article}
\usepackage{amssymb,latexsym,amsmath,amsbsy,amsthm,amsxtra,amsgen, graphicx}
\newfont{\msbm}{msbm10 at 11pt}
\newcommand {\R} {\mbox{\msbm R}}

\newcommand {\N} {\mbox{\msbm N}}

\newtheorem{Theo}{Theorem}
\newtheorem{Lemma}[Theo]{Lemma}

\newtheorem{Prop}[Theo]{Proposition}

\newtheorem{Rmk}[Theo]{Remark}

\def\eps{\varepsilon}

\begin{document}
\title{Dynamics of the evolving Bolthausen-Sznitman coalescent}
\author{Jason Schweinsberg\thanks{Supported in part by NSF Grant DMS-0805472} \\
University of California at San Diego}
\maketitle



\vspace{-.6in}
\begin{abstract}
Consider a population of fixed size that evolves over time.  At each time, the genealogical structure of the population can be described by a coalescent tree whose branches are traced back to the most recent common ancestor of the population.  As time goes forward, the genealogy of the population evolves, leading to what is known as an evolving coalescent.  We will study the evolving coalescent for populations whose genealogy can be described by the Bolthausen-Sznitman coalescent.  We obtain the limiting behavior of the evolution of the time back to the most recent common ancestor and the total length of the branches in the tree.  By similar methods, we also obtain a new result concerning the number of blocks in the Bolthausen-Sznitman coalescent.
\end{abstract}

\section{Introduction}

Consider a haploid population of fixed size $n$ that evolves over time.  The genealogy of the population at time $t$ can be represented by a coalescent process $(\Pi(s), s \geq 0)$ taking its values in the set of partitions of $\{1, \dots, n\}$, which is defined so that integers $i$ and $j$ are in the same block of $\Pi(s)$ if and only if the $i$th and $j$th individuals in the population at time $t$ have the same ancestor at time $t - s$.  The genealogical structure encoded by the coalescent process can also be represented as a tree ${\cal T}_n(t)$.  The shape of this tree changes over time as the population evolves, leading to what was called in \cite{pw06} an evolving coalescent.  The associated tree-valued stochastic process, for infinite as well as finite populations, was constructed and studied by Greven, Pfaffelhuber, and Winter \cite{gpw08}.  Depperschmidt, Greven, and Pfaffelhuber \cite{dgp11} incorporated mutation and selection into the model.

Rather than studying the full tree-valued process, one can follow the evolution of certain properties of the tree that are of interest.  One such property is the time back to the most recent common ancestor (MRCA) of the population.  This evolution can be described by a process $(A_n(t), t \geq 0)$, where $$t - A_n(t) = \sup\{s: \mbox{ one individual at time }s\mbox{ is the ancestor of all individuals at time }t\}.$$  Note that $A_n(t)$ is the height of the tree ${\cal T}_n(t)$.  The process $(A_n(t), t \geq 0)$ increases linearly at speed one between jumps, and jumps downward when one of the two oldest families in the population dies out, causing a new MRCA to be established.  One can also consider the process $(L_n(t), t \geq 0)$, where $L_n(t)$ denotes the sum of the lengths of all branches in the tree ${\cal T}_n(t)$.  This process is of interest because, assuming that mutations occur at a constant rate $\theta$ along each branch of the coalescent tree, $L_n(t)$ should be roughly proportional to the number of distinct mutations observed in the population at time $t$.

A natural population model to consider is the Moran model \cite{moran}.  In this model, the population size stays fixed at $n$.  Changes in the population occur at times of a homogeneous Poisson process, and we will scale time so that these changes occur at rate $n(n-1)/2$.  At the time of each such change, one of the $n$ individuals is chosen at random to give birth to a new offspring, and independently one of the other $n-1$ individuals is chosen at random to be killed.  For any fixed $t \in \R$, the genealogy of the population follows Kingman's coalescent \cite{king82}, meaning that each pair of lineages merges at rate one and no other transitions are possible.  An analogous construction for infinite populations can be carried out using the lookdown construction of Donnelly and Kurtz \cite{dk99}.  The associated evolving coalescent was studied by Pfaffelhuber and Wakolbinger \cite{pw06}.  They showed that the jumps of the process $(A(t), t \geq 0)$ that follows the time back to the MRCA occur at times of a homogeneous Poisson process, but that the process $(A(t), t \geq 0)$ is not Markov.  They also calculated the distributions of some other quantities, such as the number of individuals in the population at time $t$ that will have descendants in the population when the next MRCA is established and the number of individuals in the population that will become the MRCA of the population in the future.  Delmas, Dhersin, and Siri-Jegousse \cite{dds10} extended these results by considering also the distribution of the sizes of the two oldest families at time $t$.  Simon and Derrida \cite{sd06} did some further work related to the evolution of the MRCA for populations with genealogies governed by Kingman's coalescent, and considered correlations between the time back to the MRCA and a measure of genetic diversity.  Pfaffelhuber, Wakolbinger, and Weisshaupt \cite{pww09} studied the evolution of the total branch length.  They showed that the sequence of processes $(L_n(t) - 2 \log n, t \in \R)$ converges as $n \rightarrow \infty$ in the Skorohod topology to a limit process which is a stationary process with infinite infinitesimal variance.

Evans and Ralph \cite{er10} studied the dynamics of the time back to the MRCA in a population in which a single ``immortal particle" produces offspring at times of a Poisson process, and descendants of the offspring eventually die out.  In this setting, the process $(A(t), t \geq 0)$ is a Markov process whose jump rates and stationary distribution can be calculated explicitly.  An example of a process that fits into this framework is the $\alpha$-stable continuous-state branching process conditioned on nonextinction with $1 < \alpha \leq 2$.

The goal of the present paper is to determine the dynamics of the time back to the MRCA and the total branch length for populations whose genealogy is given by the Bolthausen-Sznitman coalescent.  The Bolthausen-Sznitman coalescent, which was introduced in \cite{bosz98}, is an example of a coalescent with multiple mergers \cite{pit99, sag99}, in which it is possible for many lineages to merge at once.  More precisely, the Bolthausen-Sznitman coalecent started with $n$ blocks is a continuous-time Markov chain taking its values in the set of partitions of $\{1, \dots, n\}$ such that whenever the partition has $b$ blocks, each transition that involves the merger of $k$ blocks into one is happening at rate
\begin{equation}\label{lambk}
\lambda_{b,k} = \int_0^1 x^{k-2} (1-x)^{b-k} \: dx = \frac{(k-2)! (b-k)!}{(b-1)!}.
\end{equation}
This means that the total rate of all transitions when there are $b$ blocks is $$\lambda_b = \sum_{k=2}^b \binom{b}{k} \lambda_{b,k}.$$  It is well-known (see, for example, \cite{pit99}) that this process has the property of sampling consistency, meaning that if $m < n$, then the process restricted to the integers $\{1, \dots, m\}$ is a Bolthausen-Sznitman coalescent started with $m$ blocks.

The reason for focusing on the Bolthausen-Sznitman coalescent is that this coalescent process has arisen in a wide variety of settings.  The Bolthausen-Sznitman coalescent has been shown recently to describe the genealogy of certain populations undergoing selection \cite{bbm1, bdmm2}.  The Bolthausen-Sznitman coalescent also describes the genealogical structure of Neveu's continuous-state branching process \cite{beleg00}, certain Galton-Watson processes with heavy-tailed offspring distributions \cite{sch03}, and Derrida's generalized random energy model \cite{bosz98, bovkur}.  

We describe the population model that we will study in subsection 1.1.  In subsection 1.2, we state our main result concerning the dynamics of the time back to the MRCA.  In subsection 1.3, we state our main result concerning the total branch length.  By similar methods, we also obtain a new result concerning the number of blocks of the Bolthausen-Sznitman coalescent.  We state this result in subsection 1.4.  The rest of the paper is devoted to proofs.

\subsection{A population model}

We now define more precisely the population model that we will study in this paper.  We assume that for all times $t \in \R$, there are exactly $n$ individuals in the population, labeled by the integers $1, \dots, n$.  Changes in the population occur at times of a homogeneous Poisson process on $\R$ with rate $n-1$.  At the time of such a change, one particle is chosen at random to give birth to a random number $\xi$ of new offspring, with
\begin{equation}\label{popmodel}
P(\xi = k) = \frac{n}{n-1} \cdot \frac{1}{k(k+1)}, \hspace{.1in}k = 1, 2, \dots, n-1.
\end{equation}
Then $\xi$ of the $n-1$ individuals who did not give birth are chosen at random to be killed, and the new individuals take over the labels of the individuals who were killed.  

We now give a representation of the genealogy of this population.  For each $s \in \R$ and $t \geq 0$, let $\Pi_n(s,t)$ be the partition of $1, \dots, n$ such that $i$ and $j$ are in the same block of $\Pi_n(s,t)$ if and only if the individuals at time $s$ labeled $i$ and $j$ are descended from the same ancestor immediately before time $s - t$.  We consider the population immediately before time $s-t$, rather than exactly at time $s-t$, to ensure that for each $s \in \R$, the process $(\Pi_n(s,t), t \geq 0)$ is right continuous.  As long as there is no change in the population at time $s$, the partition $\Pi_n(s,0)$ consists of $n$ singletons.  However, if $k$ new individuals are born at time $s$, then $\Pi_n(s,0)$ will consist of one block of size $k+1$ and $n-k-1$ singleton blocks.  

\begin{Prop}\label{boszprop}
Fix $s \in \R$.  Then $(\Pi_n(s,t), t \geq 0)$ is the Bolthausen-Sznitman coalescent started with $n$ blocks.
\end{Prop}

\begin{proof}
With probability one, there is no change in the population at time $s$, and $\Pi_n(s,0)$ consists of $n$ singletons.
When ancestral lines are followed backwards in time, an event in which $k-1$ individuals are born becomes an event in which $k$ randomly chosen lineages merge, because the ancestral lines of the $k-1$ children merge with that of the parent.  The rate of events in which $k-1$ individuals are born is
\begin{equation}\label{lambform}
(n-1) P(\xi = k-1) = \frac{n}{k(k-1)} = \binom{n}{k} \lambda_{n,k},
\end{equation}
where $\lambda_{n,k}$ comes from (\ref{lambk}).  Therefore, $(\Pi_n(s,t), t \geq 0)$ follows the dynamics of the Bolthausen-Sznitman coalescent up to the time of the first merger.  That the process continues to follow the dynamics of the Bolthausen-Sznitman coalescent after this time is a consequence of the exchangeability of the population model and the sampling consistency of the Bolthausen-Sznitman coalescent.
\end{proof}

For $s \in \R$ and $t \geq 0$, let $N_n(s,t)$ be the number of blocks of the partition $\Pi_n(s,t)$.  Let $$A_n(s) = \inf\{t: N_n(s,t) = 1\},$$ which is the time back to the MRCA of the population and corresponds to the height of the tree ${\cal T}_n(s)$ that represents the genealogy of the population at time $s$.  Let $$L_n(s) = \int_0^{\infty} N_n(s,t) {\bf 1}_{\{N_n(s,t) > 1\}},$$ which is sum of the lengths of all branches in the tree ${\cal T}_n(s)$.

\subsection{Time back to the MRCA}

We consider here how the time back to the MRCA of the population evolves over time.  Proposition 3.4 of \cite{gm05} states that if $Y$ has the exponential distribution with mean 1, then for each fixed $t \in \R$, we have
\begin{equation}\label{Anconv}
A_n(t) - \log \log n \Rightarrow - \log Y,
\end{equation}
where $\Rightarrow$ denotes convergence in distribution as $n \rightarrow \infty$.  We are interested here in finding the limit of the stochastic processes $(A_n(t), t \geq 0)$ as $n \rightarrow \infty$.

We now construct the limit process $(A(t), t \in \R)$ and its time-reversal $(R(t), t \in \R)$ from a Poisson point process.  The construction is similar to constructions in \cite{epw06} and \cite{er10}.  Let $N$ be a Poisson process on $\R^2$ with intensity $\lambda \times \nu$, where $\lambda$ is Lebesgue measure and $\nu(dy) = e^{-y} \: dy$.  Let $M$ consist of the points $\{(t,y): (-t,y) \mbox{ is a point of }N\}$.  Note that $M$ is also a Poisson process on $\R^2$ with the same intensity as $N$.  For each $(t,x) \in \R^2$, define the wedges
\begin{align}
W(t,x) &= \{(s,y) \in \R^2: s \leq t \mbox{ and }y \geq x + t-s\}, \nonumber \\
W'(t,x) &= \{(s,y) \in \R^2: s > t \mbox{ and }y \geq x + s - t\}. \nonumber
\end{align}
Then let
\begin{align}
R(t) &= \sup\{x: \mbox{ there is a point of }N \mbox{ in }W(t,x)\}, \nonumber \\
A(t) &= \sup\{x: \mbox{ there is a point of }M \mbox{ in }W'(t,x)\}. \nonumber
\end{align}
Note that with this construction, both $(R(t), t \in \R)$ and $(A(t), t \in \R)$ are right continuous, and
because of the relationship between $N$ and $M$, we have $A(t) = R(-t)$ for all $t$ such that there is no point in $M$ whose first coordinate is $t$.

The figure below shows how $(R(t), t \in \R)$ and $(A(t), t \in \R)$ are constructed from the Poisson point process.  The process $(R(t), t \in \R)$ decreases linearly at speed one between jumps but jumps up to the level of any point of $N$ that appears above it.  That is, if $(t,y)$ is a point of $N$ and $R(t-) < y$, then $R(t) = y$.  The process $(A(t), t \in \R)$ increases linearly at speed one between jumps.  When the trajectory of the process encounters a point of $M$ at time $t$, the process jumps downward to the highest level $y$ such that there is a point of $M$ on the diagonal half-line starting at $(t,y)$ and extending upward and to the right.

\begin{figure}
\hspace{-1.1in}\includegraphics{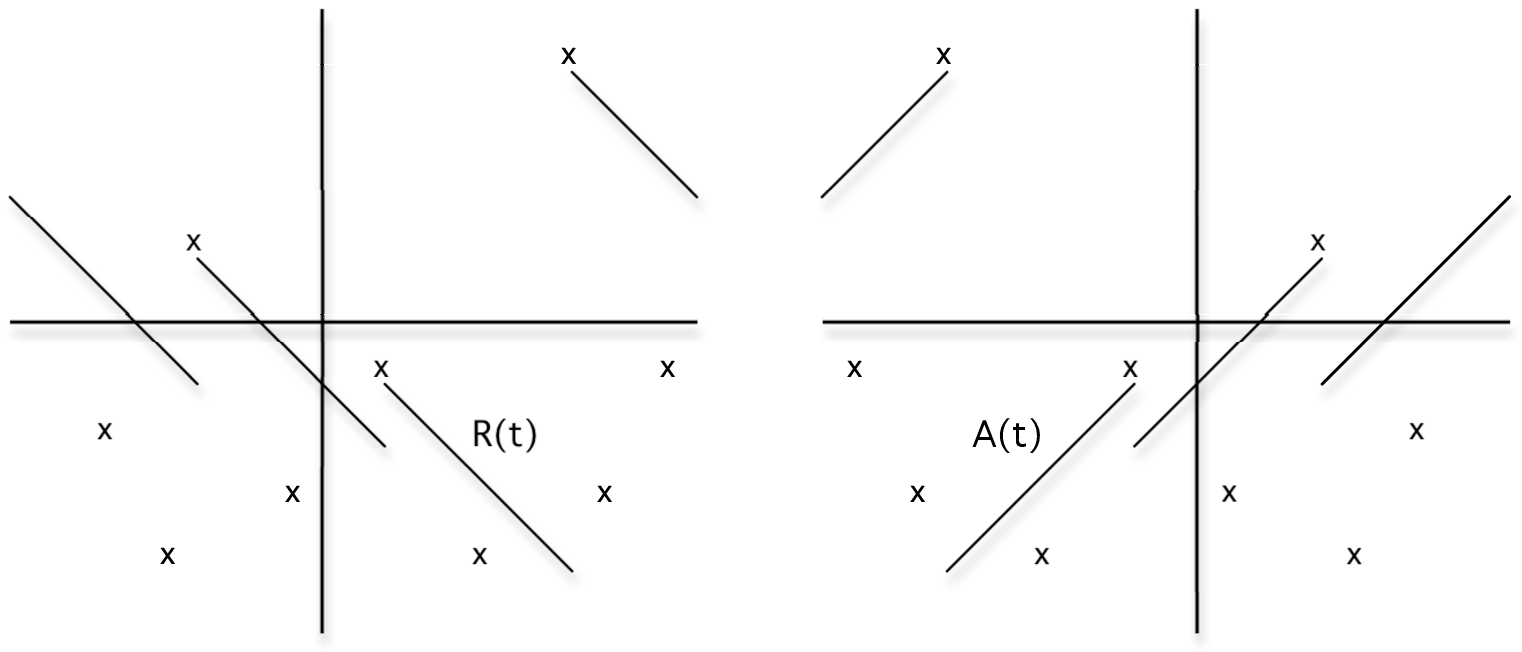}
\caption{}
\vspace{-8.8in}
\vspace{.5in}
\hspace{1.2in} Figure 1: The processes $(R(t), t \in \R)$ and $(A(t), t \in \R)$
\end{figure}

The following theorem is our main result concerning the time back to the MRCA for the evolving Bolthausen-Sznitman coalescent.  This result is proved in Section \ref{MSec} by considering the processes in reversed time and establishing convergence to $(R(t), t \geq 0)$.  Note that unlike in the case of Kingman's coalescent, the limit process is Markov.

\begin{Theo}\label{MRCAthm}
As $n \rightarrow \infty$, the sequence of processes $((A_n(t) - \log \log n), t \geq 0)$ converges in the Skorohod topology to $(A(t), t \geq 0)$.
\end{Theo}

\begin{Rmk}
{\em For each fixed $t \geq 0$, the event that $R(t) \leq y$ differs from the event that there is no point of $N$ in $W(t,y)$ only on a set of probability zero.  The probability that there is no point of $N$ in $W(t,y)$ is
\begin{equation}\label{statprob}
\exp \bigg(-\int_{-\infty}^t \int_{y+t-s}^{\infty} e^{-x} \: dx \: ds \bigg) = \exp(-e^{-y}).
\end{equation}
If $Y$ has the exponential distribution with mean one, then $P(-\log Y \leq y) = P(Y \geq e^{-y}) = \exp(-e^{-y})$.  Thus, we see from the above construction that the processes $(R(t), t \in \R)$ and $(A(t), t \in \R)$ are stationary processes whose stationary distribution is the same as the distribution of $-\log Y$.  This result also follows from Theorem \ref{MRCAthm} and equation (\ref{Anconv}).}
\end{Rmk}

\begin{Rmk}
{\em The processes $(R(t), t \in \R)$ and $(A(t), t \in \R)$ are both examples of piecewise deterministic Markov processes, a class of processes whose theory was developed by Davis \cite{dav1, dav2}.  These processes are characterized by their deterministic behavior between jump times, which is linear drift for the processes $(R(t), t \in \R)$ and $(A(t), t \in \R)$, and their jump rates.  The jump rates for $(R(t), t \in \R)$ can easily be read from the Poisson process $N$.  The process $(R(t), t \in \R)$ jumps away from $x$ at rate $e^{-x}$, and when it jumps away from $x$, the distribution of the location to which it jumps has density $e^{x-y} {\bf 1}_{\{y \geq x\}}$.  This means that the rate of jumps from $x$ to $y$ is given by $q(x,y) = e^{-y} {\bf 1}_{\{y \geq x\}}$.

To obtain the jump rates for $(A(t), t \in \R)$, note that if $A(t) = x$ then there is a point of $M$ at $(t+s, x+s)$ for some $s \geq 0$ but no points above this diagonal line.  Because the density of the intensity measure of $M$ at $(t+s, x+s)$ is $e^{-(x+s)}$, we see that conditional on $A(t) = x$, the distribution of the time before the next jump is exponential with mean $1$.  That is, for all $x \in \R$, the process $(A(t), t \in \R)$ jumps away from $x$ at rate one, which implies that the jump times of $(A(t), t \in \R)$ form a homogeneous Poisson process of rate one on $\R$.  If the process jumps away from $x$ at time $t$, then the probability that it jumps below $y$ is the probability that there is no point of $M$ in the trapezoidal region $\{(s,z): s \geq t \mbox{ and } y + s - t \leq z \leq x + s - t\}$, which is $$\exp \bigg( -\int_t^{\infty} \int_{y+s-t}^{x+s-t} e^{-z} \: dz \: ds \bigg) = \exp(e^{-x} - e^{-y}).$$  Differentiating with respect to $y$, we see that the rate of jumps from $x$ to $y$ is given by $r(x,y) = \exp(e^{-x} - e^{-y} - y) {\bf 1}_{\{y \leq x\}}$.

As a check on these formulas, let $\pi(y) = e^{-y} e^{-e^{-y}}$ for $y \in \R$, which is the density of the stationary distribution, obtained by differentiating the right-hand side of (\ref{statprob}).  Then note that $\pi(x) q(x,y) = \pi(y) r(y,x)$ for all $x,y \in \R$, as expected given that $(R(t), t \in \R)$ and $(A(t), t \in \R)$ are related by time reversal.
}
\end{Rmk}

\subsection{Total branch length}

Theorem 5.2 of \cite{dimr07} establishes that for each fixed $t \in \R$, 
\begin{equation}\label{dimr}
\frac{(\log n)^2}{n} \bigg(L_n(t) - \frac{n}{\log n} - \frac{n \log \log n}{(\log n)^2} \bigg) \Rightarrow X,
\end{equation}
where, using $\gamma$ to denote Euler's constant,
\begin{align}\label{chfX}
E[e^{iuX}] &= \exp \bigg( - \frac{\pi}{2} |u| + i u \log |u| \bigg) \nonumber \\
&= \exp \bigg(iu(1 - \gamma) - \int_{-\infty}^0 \big(1 - e^{iux} + iux {\bf 1}_{\{|x| \leq 1\}} \big) x^{-2} \: dx \bigg).
\end{align}

We consider here the stochastic process $L_n = (L_n(t), t \geq 0)$.  If there are no changes in the population between times $t$ and $t + s$, then the tree ${\cal T}_n(t+s)$ is obtained by starting with the tree ${\cal T}_n(s)$ and then adding a segment of length $s$ to each of the $n$ branches.  Consequently, the process $L_n$ increases at speed $n$ between jumps.  However, if $k$ individuals die at time $t$, then the tree ${\cal T}_n(t)$ is obtained from the tree ${\cal T}_n(t-)$ by removing $k$ of the branches, causing a downward jump in the process $L_n$.  Our main result concerning the dynamics of the total branch length for the Bolthausen-Sznitman coalescent is that the processes $(L_n(t), t \geq 0)$, properly centered and scaled, converge to a stable process of Ornstein-Uhlenbeck type.

We now review some facts about processes of Ornstein-Uhlenbeck type, which can be found in Chapter 17 of \cite{sato}.
Suppose $(Z(t), t \geq 0)$ is a L\'evy process such that $$E[e^{iuZ(t)}] = \exp \bigg(iaut - \frac{bu^2t}{2} - t \int_{-\infty}^{\infty} \big(1 - e^{iux} + iux {\bf 1}_{\{|x| \leq 1\}}\big) \: \nu(dx) \bigg).$$  Given $c > 0$ and a random variable $X(0)$, there is a unique process $(X(t), t \geq 0)$ having paths that are almost surely right continuous with left limits such that
\begin{equation}\label{OUSDE}
X(t) = X(0) + Z(t) - c \int_0^t X(s) \: ds
\end{equation}
almost surely.  Following \cite{sato}, we call $(X(t), t \geq 0)$ the process of Ornstein-Uhlenbeck type generated by $(a, b, \nu, c)$.
As long as $$\int_{|x| > 2} \log |x| \: \nu(dx) < \infty,$$ the process $(X(t), t \geq 0)$ has a unique stationary distribution $\mu$, with characteristic function
\begin{equation}\label{chfmu}
\int_{-\infty}^{\infty} e^{iux} \mu(dx) = \exp \bigg(i \alpha u - \frac{\beta u^2}{2} - \int_{-\infty}^{\infty} \big(1 - e^{iux} + iux {\bf 1}_{\{|x| \leq 1\}} \big) \: \rho(dx) \bigg),
\end{equation}
where $\beta = b/2c$, $\alpha = (a + \nu((1, \infty)) - \nu((-\infty, -1)))/c$, and for all Borel sets $B$, $$\rho(B) = \frac{1}{c} \int_{-\infty}^{\infty} \int_0^{\infty} {\bf 1}_B(e^{-s} y) \: ds \: \nu(dy)$$ (see Theorem 17.5 of \cite{sato}).  In this case, if $X(0)$ has distribution $\mu$, then the process $(X(t), t \geq 0)$ is stationary.

\begin{Theo}\label{lengththm}
Let $\nu$ be the measure on $\R$ whose density with respect to Lebesgue measure is given by $x^{-2} {\bf 1}_{(-\infty, 0)}(x)$.  Let $(L(t), t \geq 0)$ be a stationary process of Ornstein-Uhlenbeck type generated by $(2 - \gamma, 0, \nu, 1)$.  As $n \rightarrow \infty$, the sequence of processes $$\bigg(\frac{(\log n)^2}{n} \bigg(L_n\bigg( \frac{t}{\log n} \bigg) - \frac{n}{\log n} - \frac{n \log \log n}{(\log n)^2} \bigg), t \geq 0 \bigg)$$ converges in the Skorohod topology to $(L(t), t \geq 0).$
\end{Theo}

\begin{Rmk}
{\em The stationary distribution of $(L(t), t \geq 0)$ has characteristic function given by (\ref{chfmu}) with $\alpha = 1 - \gamma$, $\beta =  0$, and $\rho = \nu$, which matches the right-hand side of (\ref{chfX}).  To see this, observe that $\nu((1, \infty)) = 0$ and $\nu((-\infty, 1)) = 1$, so $\alpha = (2 - \gamma) - 1 = 1 - \gamma$.  Also, for all $z > 0$, $$\rho((-\infty, -z]) = \int_{-\infty}^0 \int_0^{\infty} {\bf 1}_{(-\infty, -z]}(e^{-s} y) y^{-2} \: ds \: dy = \int_0^{\infty} \int_{-\infty}^{-ze^s} y^{-2} \: dy \: ds = \frac{1}{z} = \nu((-\infty, -z]).$$  Thus, the convergence implied by Theorem \ref{lengththm} for each fixed $t$ is consistent with the result (\ref{dimr}).}
\end{Rmk}

\subsection{Number of blocks}

The techniques used to establish Theorem \ref{lengththm} can also be used to prove a new result about how the number of blocks of the Bolthausen-Sznitman coalescent changes over time.  Because the number of blocks for other coalescents with multiple mergers has been studied in some depth (see, for example, \cite{bbl, bbs1}), we believe that this result may be of independent interest.

\begin{Theo}\label{numblocksthm}
Let $(\Pi_n(t), t \geq 0)$ be a Bolthausen-Sznitman coalescent started with $n$ blocks.  Let $N_n(t)$ be the number of blocks of $\Pi_n(t)$, and let $$X_n(t) = \frac{\log n}{n} \bigg( N_n \bigg( \frac{t}{\log n} \bigg) - n e^{-t} - \frac{n t e^{-t} \log \log n}{\log n} \bigg).$$  Let $(S(t), t \geq 0)$ be a stable L\'evy process satisfying
$$E[e^{iuS(t)}] = \exp \bigg( - \frac{\pi t}{2} |u| + itu \log |u| \bigg).$$  As $n \rightarrow \infty$, the sequence of processes $(X_n(t), t \geq 0)$ converges in the Skorohod topology to $$\bigg( e^{-t} S(t) + \frac{e^{-t} t^2}{2}, t \geq 0 \bigg).$$
\end{Theo}

\section{Proof of Theorem \ref{MRCAthm}}\label{MSec}

\subsection{Construction from random recursive trees}\label{constrrt}

Our proof of Theorem \ref{MRCAthm} makes use of a connection between the Bolthausen-Sznitman coalescent and random recursive trees that was discovered by Goldschmidt and Martin \cite{gm05}.  Suppose $\ell_1, \dots, \ell_n$ are disjoint subsets of $\N$ such that if $1 \leq i < j \leq n$, then the smallest element of $\ell_i$ is less than the smallest element of $\ell_j$.  A random recursive tree with vertices labeled $\ell_1, \dots, \ell_n$ can be constructed inductively as follows.  The vertex labeled $\ell_1$ is the root.  For $k \geq 2$, once the vertices labeled $\ell_1, \dots, \ell_{k-1}$ have been placed in the tree, the vertex labeled $\ell_k$ is attached to a vertex chosen uniformly at random from those labeled $1, \dots, \ell_{k-1}$.

Suppose $e$ is an edge in the tree connecting vertices $x$ and $y$, where $x$ is closer to the root than $y$.  We can cut the tree at the edge $e$ by deleting the edge from the tree as well as the entire subtree below $e$.  That is, we remove all vertices $z$ such that the shortest path from the root to $z$ goes through $x$.  All integers that are in labels of vertices that are removed from the tree are then added to the label of $x$.  When a random recursive tree is cut at a randomly chosen edge, the remaining tree is a random recursive tree on the new set of labels (see Proposition 2.1 of \cite{gm05}).

To establish the connection with the Bolthausen-Sznitman coalescent, start with a random recursive tree on $n$ vertices, labeled with the integers $1, \dots, n$.  Then to each edge, add an independent exponential random variable with mean $1$, whose value gives the time at which the edge is cut.  For all $t \geq 0$, let $\Pi_n(t)$ denote the partition of $\{1, \dots, n\}$ such that $i$ and $j$ are in the same block of $\Pi_n(t)$ if and only if the integers $i$ and $j$ are in the same vertex label at time $t$.  Then $(\Pi_n(t), t \geq 0)$ is the Bolthausen-Sznitman coalescent started with $n$ blocks (see Proposition 2.2 of \cite{gm05}).  Because the last transition always involves deleting an edge adjacent to the root, the time back to the MRCA for this Bolthausen-Sznitman coalescent is the maximum of the exponential random variables assigned to the edges adjacent to the root.  In \cite{gm05}, Goldschmidt and Martin used this fact to prove (\ref{Anconv}).

We now use recursive trees to construct the evolving Bolthausen-Sznitman coalescent in reversed time.  Note that the dynamics of an evolving Bolthausen-Sznitman coalescent in reversed time are the same as the dynamics of an ordinary Bolthausen-Sznitman coalescent, except that whenever $k$ lineages are lost due to a merger, these lineages are replaced by $k$ new lineages.  For the purposes of studying the time back to the MRCA, the labeling of the lineages by the integers $1, \dots, n$ is unimportant, so we will use a different vertex labeling scheme in the recursive tree construction.

To carry out this construction, begin with a random recursive tree having $n$ vertices constructed as above, and give every vertex the label zero.  Add an independent exponential random variable with mean $1$ to each edge to obtain the tree at time zero.  The process evolves in time as follows.  The edge labels decrease linearly at speed one.  When an edge label hits zero at, say, time $t$, this edge is cut from the tree and all vertices below the edge are removed.  If this cut removes $k$ vertices, then we replace these vertices by adding $k$ new vertices to the tree, one at a time, to randomly chosen vertices of the existing tree.  The $k$ new edges are assigned independent exponential random variables.  The $k$ new vertices are given the label $t$, corresponding to the time when they are added to the tree.  Then the process continues to evolve according to the same rules.

By the result of Goldschmidt and Martin, we know that this process follows the dynamics of the Bolthausen-Sznitman coalescent up to the time of the first merger, in the sense that the time to the first transition is exponential with rate $n - 1 = \lambda_n$ and the probability that the first transition eliminates $k$ vertices is $\lambda_{n,k+1}/\lambda_n$.  To see that the same dynamics continue after the first merger, note that if the first cut causes $k$ vertices to be removed, the remaining tree has the shape of a random recursive tree on $n-k$ vertices, while the edge lengths remain exponential random variables with mean $1$ by the memoryless property of the exponential distribution.  Consequently, when $k$ more vertices are added according to the recursive procedure, the resulting tree has the shape of a random recursive tree on $n$ vertices, and all of the random variables attached to the edges have the exponential distribution with mean $1$.  Thus, this process follows the same dynamics as the population process followed backwards in time, with each set of $k$ lineages merging at rate $\lambda_{n,k}$.

Let $M_n(t)$ denote the maximum of the exponential random variables assigned to the edges that are adjacent to the root at time $t$.  Then $t + M_n(t)$ is the first time at which every vertex other than the root has a label greater than $t$, and $M_n(t)$ corresponds to the time back to the MRCA of the population at time $-t$.  Let $R_n(t) = M_n(t) - \log \log n$.  Then, we see that $(R_n(t), t \geq 0)$ has the same finite-dimensional distributions as $(A_n(-t) - \log \log n, t \geq 0)$, and the two processes would have the same law if the process $(A_n(-t), t \geq 0)$ were modified at the jump times to make it right-continuous rather than left-continuous.  Consequently, in view of the stationarity of the population process, to prove Theorem \ref{MRCAthm}, it suffices to show that the processes $(R_n(t), t \geq 0)$ converge in the Skorohod topology to $(R(t), t \geq 0)$, and it is this result that we will show.

\subsection{A heuristic argument}

To understand heuristically why Theorem \ref{MRCAthm} is true, note that if $R_n(t) \leq z$, then the process $R_n$ jumps above $z$ only when a new vertex is attached to the root, and the random variable assigned to the new edge is greater than $\log \log n + z$.  Because the number of blocks in the Bolthausen-Sznitman coalescent decreases by $k-1$ whenever $k$ blocks merge into one, the rate at which blocks are being lost, and thus new vertices are being added to the tree, is $$\gamma_n = \sum_{k=2}^n (k-1) \binom{n}{k} \lambda_{n,k} = \sum_{k=2}^n \frac{n}{k} \approx n \log n.$$  As long as not too many vertices are cut away from the tree at once, the probability that a new vertex attaches to the root is approximately $1/n$.  The probability that the exponential random variable assigned to the new edge is greater than $\log \log n + z$ is $e^{-z}/\log n$.  Hence, the rate at which the process $R_n$ jumps above $z$ is approximately $$(n \log n) \bigg( \frac{1}{n} \bigg) \bigg( \frac{e^{-z}}{\log n} \bigg) = e^{-z},$$ in agreement with the dynamics of the process $(R(t), t \geq 0)$.  The rest of the proof consists of making these ideas rigorous.

\subsection{Lemmas pertaining to random recursive trees}

We prove here two lemmas related to random recursive trees that will be used later in the proof of Theorem \ref{MRCAthm}.

\begin{Lemma}\label{jumprate}
For all $\eps > 0$ and $z \in \R$, the probability that there exists $u \in [0, \eps]$ such that $R_n(u) \neq R_n(u-)$ and $R_n(u) > z$ is at most $\eps (1 + 2 e^{-z}).$
\end{Lemma}

\begin{proof}
The rate at which blocks are being lost in the Bolthausen-Sznitman coalescent is given by
\begin{equation}\label{gamnbd}
\gamma_n = \sum_{k=2}^n \frac{n}{k} \leq n \log n.
\end{equation}
The rate of transitions in which at least half the blocks are lost is $$\sum_{k = \lceil n/2 \rceil + 1}^n \binom{n}{k} \lambda_{n,k} = \sum_{k = \lceil n/2 \rceil + 1}^n \frac{n}{k(k-1)} = n \bigg( \frac{1}{\lceil n/2 \rceil} - \frac{1}{n} \bigg) \leq 1,$$ so the probability that such a transition occurs by time $\eps$ is bounded by $\eps$.

By the construction, for each block that is lost due to a merger in the Bolthausen-Sznitman coalescent, a new vertex is added to the tree.  As long as the merger causes at most half the blocks to disappear, the probability that each new vertex attaches to the root is bounded by $2/n$.  Furthermore, when a new vertex attaches to the root, the probability that it causes the process $(R_n(t), t \geq 0)$ to jump above $z$ is the same as the probability that an exponential random variable with mean $1$ is greater than $\log \log n + z$, which is $e^{-(\log \log n + z)} = e^{-z}/\log n$.  Consequently, in view of (\ref{gamnbd}), the probability that the process $(R_n(t), t \geq 0)$ jumps above $z$ before time $\eps$ is bounded by $$\eps + \eps \gamma_n \cdot \frac{2}{n} \cdot \frac{e^{-z}}{\log n} \leq \eps (1 + 2e^{-z}),$$  which implies the result.
\end{proof}

\begin{Lemma}\label{DnLem}
Consider a random recursive tree with vertices labeled $1, \dots, n$.  Let $d_k$ be the depth of the vertex labeled $k$, which is the number of edges on the path from the root to $k$, and let $D_n = d_1 + \dots + d_n$.  Then
\begin{equation}\label{EDn}
E[D_n] = n \bigg( \sum_{k=1}^{n} \frac{1}{k} - 1 \bigg).
\end{equation}
Also, there exists a positive constant $C$ such that
\begin{equation}\label{VarDn}
\textup{Var}(D_n) \leq Cn^2.
\end{equation}
\end{Lemma}

\begin{proof}
We first prove (\ref{EDn}) by induction.  Because the vertex labeled $1$ is the root vertex, which has depth zero, clearly $E[D_1] = 0$, verifying (\ref{EDn}).  Suppose (\ref{EDn}) holds for $n = m-1$, where $m \geq 2$.  Let ${\cal F}_{m-1} = \sigma(d_1, \dots, d_{m-1})$.  Recall that the random recursive tree can be constructed so that the vertex labeled $m$ is attached to one of the previous $m-1$ vertices at random.  Consequently, the level of the vertex labeled $m$ is one greater than the level of a randomly chosen previous vertex, so
\begin{equation}\label{condedm}
E[d_m|{\cal F}_{m-1}] = 1 + \frac{d_1 + \dots + d_{m-1}}{m-1} = 1 + \frac{D_{m-1}}{m-1}
\end{equation}
and thus $$E[d_m] = 1 + \frac{E[D_{m-1}]}{m-1}.$$
Therefore, using the induction hypothesis, $$E[D_m] = E[D_{m-1}] + E[d_m] = 1 + \frac{m}{m-1} E[D_{m-1}] = 1 + m \bigg(\sum_{k=1}^{m-1} \frac{1}{k} - 1\bigg) = m \bigg(\sum_{k=1}^m \frac{1}{k} - 1  \bigg),$$ which implies that (\ref{EDn}) holds for all $n \in \N$. 

We next show by induction that
\begin{equation}\label{varDn}
\mbox{Var}(D_n) \leq n^2 \sum_{k=1}^n \frac{E[d_k^2]}{k^2}.
\end{equation}
The result is clear when $n = 1$ because $\mbox{Var}(D_1) = 0$.  Suppose the claim holds for $n = m-1$, where $m \geq 2$.  By (\ref{condedm}), $E[D_m|{\cal F}_{m-1}] = 1 + mD_{m-1}/(m-1)$, so using the induction hypothesis,
\begin{align}
\mbox{Var}(D_m) &= E[\mbox{Var}(D_m|{\cal F}_{m-1})] + \mbox{Var}(E[D_m|{\cal F}_{m-1}]) \nonumber \\
&= E[\mbox{Var}(d_m|{\cal F}_{m-1})] + \mbox{Var} \bigg( \frac{m}{m-1} D_{m-1} \bigg) \nonumber \\
&\leq E[E[d_m^2|{\cal F}_{m-1}]] + m^2 \sum_{k=1}^{m-1} \frac{E[d_k^2]}{k^2} \nonumber \\
&= m^2 \sum_{k=1}^m \frac{E[d_k^2]}{k^2}. \nonumber
\end{align}
Now (\ref{varDn}) follows by induction.

It remains to bound $E[d_k^2]$.  Suppose $k \geq 2$.  If $d_k = j \geq 2$, then there is a sequence of numbers $1 = i_0 < i_1 < \dots < i_j = k$ such that during the construction of the random recursive tree, vertex $i_{\ell}$ attaches to vertex $i_{\ell - 1}$ for $\ell = 1, \dots, j$.  Because the vertex $i_{\ell}$ has a choice of $i_{\ell} - 1$ vertices to which it can attach, the probability of this event is $1/[(i_1 - 1) \dots (i_{j-1}-1)(k-1)]$.  Thus,
\begin{align}
P(d_k = j) &= \frac{1}{k-1} \sum_{1 < i_1 < \dots < i_{j-1} < k} \frac{1}{(i_1 - 1) \dots (i_{j-1} - 1)} \nonumber \\
&\leq \frac{1}{(k-1)(j-1)!} \bigg( \sum_{i=2}^{k-1} \frac{1}{i-1} \bigg)^{j-1} \nonumber \\
&\leq \frac{(1 + \log k)^{j-1}}{(k-1)(j-1)!}. \nonumber
\end{align}
Combining this bound with the trivial bound that $P(d_k = j) \leq 1$ for $j = 1,2$ gives
\begin{align}
E[d_k^2] &= 5 + \frac{1}{k-1} \sum_{j=3}^{k-1} \frac{j^2 (1 + \log k)^{j-1}}{(j-1)!} \nonumber \\
&= 5 + \frac{(1 + \log k)^2}{k-1} \sum_{j=3}^{k-1} \frac{j^2}{(j-1)(j-2)} \cdot \frac{(1 + \log k)^{j-3}}{(j-3)!} \nonumber \\
&\leq 5 + \frac{9 (1 + \log k)^2}{2(k-1)} e^{1 + \log k} \nonumber \\
&\leq C (\log k)^2
\end{align}
for some positive constant $C$.  Combining this bound with (\ref{varDn}) gives (\ref{VarDn}).
\end{proof}

\subsection{Generator of the limit process}

From the Poisson process construction described in the introduction, it is clear that $(R(t), t \geq 0)$ is a Markov process.  Furthermore, it is easy to describe the transition semigroup of the Markov process.  Suppose $R(0) = x$.  Then for $y > x - t$, we have $R(t) \leq y$ when there are no points in the Poisson process $N$ above the line segment from $(0,y+t)$ to $(t,y)$.  It follows that $$P(R(t) \leq y) = \exp \bigg( - \int_0^t \int_{y+t-s}^{\infty} e^{-z} \: dz \: ds \bigg) = \exp \big(-e^{-y}(1 - e^{-t}) \big),$$ and $$P(R(t) = x - t) = \exp \big( -e^{-(x-t)}(1 - e^{-t}) \big).$$

Let $E = [-\infty, \infty)$, and for $x, y \in E$, let $d(x,y) = |e^x - e^y|$.  Then $(E, d)$ is a complete separable metric space.  Let $C_0(E)$ be the set of continuous real-valued functions on $[-\infty, \infty)$ that vanish at infinity with the norm $\|f\| = \sup_{x \in E} |f(x)|$.  Note that if $f \in C_0(E)$, then $\lim_{x \rightarrow -\infty} f(x)$ exists and equals $f(-\infty)$, and $\lim_{x \rightarrow \infty} f(x) = 0$.  If $f \in C_0(E)$ and $x \in E$, define
\begin{equation}\label{Pt}
P_tf(x) = \int_{x-t}^{\infty} f(y) e^{-e^{-y}(1 - e^{-t})} e^{-y}(1 - e^{-t}) \: dy + f(x - t) e^{-e^{-(x-t)}(1 - e^{-t})},
\end{equation}
so that $E[f(R(t))] = P_t f(x)$ when $R(0) = x$.  Note that the definition of $P_tf(x)$ makes sense when $x = -\infty$, in which case the second term is zero.  It is easily checked that $P_tf \in C_0(E)$ and that $P_t f \rightarrow f$ as $t \rightarrow 0$.  Consequently, $(P_t)_{t \geq 0}$ is a Feller semigroup on $C_0(E)$, and $(R(t), t \geq 0)$ is a Feller process with semigroup $(P_t)_{t \geq 0}$.

The following result characterizes the infinitesimal generator of the process $(R(t), t \geq 0)$ and describes a core for the generator.  We see from the form of the generator in (\ref{genA}) that for $y > x$, the process jumps from $x$ to $y$ at rate $e^{-y}$.  See also chapter 26 of \cite{dav2} for a full characterization of the domain of the extended generator.

\begin{Lemma}
Let $A$ be the infinitesimal generator associated with $(P_t)_{t \geq 0}$.  Let ${\cal C}$ be the collection of functions $f$ 
that are constant on $[-\infty, z]$ for some $z > -\infty$ and have the property that $f$, $f'$, and $f''$ are in $C_0(E)$ when we define $f'$ and $f''$ in the usual way on $(-\infty, \infty)$ and set $f'(-\infty) = f''(-\infty) = 0$.  Then for all $f \in {\cal C}$ and all $x \in E$,
\begin{equation}\label{genA}
Af(x) = -f'(x) + \int_x^{\infty} e^{-y}(f(y) - f(x)) \: dy.
\end{equation}
Furthermore, ${\cal C}$ is a core for $A$.
\end{Lemma}

\begin{proof}
Fix $f \in {\cal C}$.  Choose a real number $z$ such that $f$ is constant on $[-\infty, z]$.  By (\ref{Pt}), for all $x \in E$,
\begin{align}\label{3terms}
\frac{P_t f(x) - f(x)}{t} &= \frac{(f(x-t) - f(x))e^{-e^{-(x-t)}(1-e^{-t})}}{t} \nonumber \\
&\hspace{.2in}+ \frac{1}{t} \int_{x-t}^{\infty} (f(y) - f(x)) e^{-e^{-y}(1 - e^{-t})} e^{-y}(1 - e^{-t}) \: dy.
\end{align}
We need to show that the right-hand side of (\ref{3terms}) converges to the right-hand side of (\ref{genA}) uniformly in $x$ as $t \rightarrow 0$.  Because $f''$ is bounded, $$\frac{f(x-t) - f(x)}{t} \rightarrow -f'(x)$$ uniformly in $x$ as $t \rightarrow 0$.  Also, $$\bigg| \frac{f(x-t) - f(x)}{t} \big(e^{-e^{-(x-t)}(1-e^{-t})} - 1 \big) \bigg| \leq \bigg| \frac{f(x-t) - f(x)}{t} \bigg| e^{-(x-t)} t \leq \|f'\| t e^t e^{-z},$$ which tends to zero uniformly in $x$ as $t \rightarrow 0$.  Because
$$\bigg| 1 - \frac{e^{-e^{-y}(1 - e^{-t})}(1 - e^{-t})}{t} \bigg| \leq \big|1 - e^{-e^{-y}(1-e^{-t})}\big| + e^{-e^{-y}(1 - e^{-t})} \bigg|1 - \frac{1 - e^{-t}}{t} \bigg| \leq e^{-y}t + \frac{t}{2},$$ we have
$$\bigg| \int_x^{\infty} (f(y) - f(x)) e^{-y} \bigg( 1 - \frac{e^{-e^{-y}(1 - e^{-t})} (1 - e^{-t})}{t}\bigg) \: dy \bigg| \leq \int_z^{\infty} 2 \|f\| e^{-y}\bigg(e^{-y}t + \frac{t}{2} \bigg) \: dy,$$ which tends to zero uniformly in $x$ as $t \rightarrow 0$.  Likewise,
$$\frac{1}{t} \int_{x-t}^x (f(y) - f(x)) e^{-e^{-y}(1 - e^{-t})} e^{-y}(1 - e^{-t}) \: dy$$ tends to zero uniformly in $x$ as $t \rightarrow 0$.  Therefore,
$$\frac{1}{t} \int_{x-t}^{\infty} (f(y) - f(x)) e^{-e^{-y}(1 - e^{-t})} e^{-y}(1 - e^{-t}) \: dy \rightarrow \int_x^{\infty} e^{-y}(f(y) - f(x)) \: dy$$ uniformly in $x$ as $t \rightarrow 0$.  Equation (\ref{genA}) follows.

It remains to show that ${\cal C}$ is a core for $A$.  It is easy to see that ${\cal C}$ is dense in $C_0(E)$.  Suppose $f \in {\cal C}$, and choose $z$ so that $f$ is constant on $[-\infty, z]$.  For all $x \leq z + t$, $$P_tf(x) = \int_z^{\infty} f(y) e^{-e^{-y}(1-e^{-t})} e^{-y}(1-e^{-t}) \: dy + f(z) e^{-e^{-z}(1-e^{-t})}.$$  Thus, $P_tf$ is constant on $[-\infty, z+t]$.  By differentiating the right-hand side of (\ref{Pt}), we see that the first and second derivatives of $P_t f$ are continuous and vanish at infinity.  
Thus, $P_tf \in {\cal C}$.  It follows from Proposition 3.3 in Chapter 1 of \cite{ek86} that ${\cal C}$ is a core for $A$.
\end{proof}

\subsection{Convergence of finite-dimensional distributions}

We will show here that the finite-dimensional distributions of the processes $(R_n(t), t \geq 0)$ defined from random recursive trees in section \ref{constrrt} converge as $n \rightarrow \infty$ to the finite-dimensional distributions of $(R(t), t \geq 0)$.  

Let ${\cal G}_n(t)$ be the $\sigma$-field generated by the shape of the random recursive tree on $n$ vertices at time $t$ and the exponential random variables attached to the edges adjacent to the root.  That is, ${\cal G}_n(t)$ includes all the information about the tree at time $t$ except for the values of the exponential random variables on the edges that are not adjacent to the root.  Let ${\cal F}_n(t) = \sigma({\cal G}_n(s), 0 \leq s \leq t)$.  Note that $R_n(t)$ is ${\cal G}_n(t)$-measurable for all $t \geq 0$, and thus the process $(R_n(t), t \geq 0)$ is adapted to the filtration ${\cal F}_n = ({\cal F}_n(t), t \geq 0)$.  Because $(R_n(t), t \geq 0)$ is right continuous, it follows that $(R_n(t), t \geq 0)$ is ${\cal F}_n$-progressive (see p. 50 of \cite{ek86}).  

Fix a function $f \in {\cal C}$.  Let $\eps_n = n^{-4}$ for all $n \in \N$.  Let $$\xi_n(t) = \frac{1}{\eps_n} \int_0^{\eps_n} E[f(R_n(t+s))|{\cal F}_n(t)] \: ds$$ and $$\varphi_n(t) = \frac{1}{\eps_n} E[f(R_n(t+\eps_n)) - f(R_n(t))|{\cal F}_n(t)].$$  By (\ref{Anconv}), $R_n(0)$ converges in distribution to $R(0)$ as $n \rightarrow \infty$.  By Theorem 8.2 in chapter 4 of \cite{ek86} (see also parts (a) and (b) of Remark 8.3), to show that the finite-dimensional distributions of $(R_n(t), t \geq 0)$ converge to those of $(R(t), t \geq 0)$, it suffices to show that the following hold for all $t \geq 0$:
\begin{align}
&\sup_n \sup_{0 \leq s \leq t} E[|\xi_n(s)|] < \infty, \label{cond1} \\
&\sup_n \sup_{0 \leq s \leq t} E[|\varphi_n(s)|] < \infty, \label{cond2} \\
&\lim_{n \rightarrow \infty} E [|\xi_n(t) - f(R_n(t))|] = 0, \label{cond3} \\
&\lim_{n \rightarrow \infty} E [|\varphi_n(t) - (Af)(R_n(t))|] = 0. \label{cond4}
\end{align}

Note that (\ref{cond1}) is obvious because $|\xi_n(s)| \leq \|f\|$ for all $s \geq 0$.  To show (\ref{cond2}), choose $z > -\infty$ such that $f$ is constant on $[-\infty, z]$.  For $s \geq 0$, let $J_s$ be the event that there exists a time $u \in [s, s+\eps_n]$ such that $R_n(u) \neq R_n(u-)$ and $R_n(u) > z$.  Because the process $(R_n(s), s \geq 0)$ decreases at speed one between jumps, we have $f(R_n(s + \eps_n)) - f(R_n(s)) \leq \eps_n \|f'\|$ on the event $J_s^c$.  Because the process $(R_n(s), s \geq 0)$ is stationary, we have $P(J_s) = P(J_0)$ for all $s \geq 0$.  Therefore, by Lemma \ref{jumprate},
\begin{align}
|\varphi_n(s)| &\leq \frac{1}{\eps_n} E[|f(R_n(s + \eps_n)) - f(R_n(s))|{\bf 1}_{J_s^c}] + \frac{1}{\eps_n} E[|f(R_n(s + \eps_n)) - f(R_n(s))|{\bf 1}_{J_s}] \nonumber \\
&\leq \|f'\| + \frac{2}{\eps_n} \|f\| P(J_s) \nonumber \\
&\leq \|f'\| + (2 + 4 e^{-z}) \|f\|, \nonumber
\end{align}
which proves (\ref{cond2}).

Next, observe that $$|\xi_n(t) - f(R_n(t))| = \frac{1}{\eps_n} \bigg| \int_0^{\eps_n} E[f(R_n(t+s)) - f(R_n(t))|{\cal F}_n(t)] \: ds \bigg| \leq \eps_n \|f'\| + 2 \|f\| P(J_t|{\cal F}_n(t)).$$  Thus, by Lemma \ref{jumprate}, taking expectations of both sides gives
$$E[|\xi_n(t) - f(R_n(t))|] \leq \eps_n \|f'\| + 2 \|f\| P(J_t) \leq \eps_n \|f'\| + 2 \eps_n \|f\|(1 + 2e^{-z}) \rightarrow 0$$ as $n \rightarrow \infty$, which gives (\ref{cond3}).  It remains only to show (\ref{cond4}). 

When the tree is cut, we call the event a small cut if fewer than $n/(\log n)^{1/2}$ vertices are removed as a result of the cut, and a large cut otherwise.  We define the following five events.  Recall that $z$ has been chosen so that $f$ is constant on $[-\infty, z]$.
\begin{itemize}
\item Let $A_1$ be the event that between times $t$ and $t + \eps_n$, there is a small cut at some edge not adjacent to the root, one of the new edges attaches to the root and is assigned a label greater than $\log \log n + z$, and this is the only edge that attaches to the root between times $t$ and $t + \eps_n$ and gets a label greater than $\log \log n + z$.

\item Let $A_2$ be the event that between times $t$ and $t + \eps_n$, there is a large cut during which one of the new edges attaches to the root and is assigned a label greater than $\log \log n + z$.

\item Let $A_3$ be the event that between times $t$ and $t + \eps_n$, there is an event in which the tree is cut at some edge adjacent to the root, and one of the new edges attaches to the root and is assigned a label greater than $\log \log n + z$.

\item Let $A_4$ be the event that between times $t$ and $t + \eps_n$, two or more new edges attach to the root and are assigned labels greater than $\log \log n + z$.

\item Let $A_5 = A_1^c \cap A_2^c \cap A_3^c \cap A_4^c$.
\end{itemize}
The next lemma shows that $A_2$, $A_3$, and $A_4$ are unlikely to occur, which means that jumps of the process $R_n$ between times $t$ and $t + \eps_n$ will occur primarily on the event $A_1$.

\begin{Lemma}\label{badevents}
We have $$\lim_{n \rightarrow \infty} \frac{P(A_2 \cup A_3 \cup A_4)}{\eps_n} = 0.$$
\end{Lemma}

\begin{proof}
To bound $P(A_2)$, note that an event during which $k-1$ vertices are removed from the tree corresponds to a transition in the Bolthausen-Sznitman coalescent in which $k$ blocks merge into one.  Such events happen at rate $\binom{n}{k} \lambda_{n,k} = n/[k(k-1)]$ by (\ref{lambform}).  When such an event occurs, the expected number of vertices that reattach to the root is $1/(n-k+1) + \dots + 1/(n-1)$.  When a vertex reattaches to the root, the probability that its label exceeds $\log \log n + z$ is $e^{-(\log \log n + z)} = e^{-z}/\log n$.  Thus, for sufficiently large $n$,
\begin{align}\label{PA2}
P(A_2) &\leq \eps_n \sum_{k = \lceil n/(\log n)^{1/2} + 1 \rceil}^n \frac{n}{k(k-1)} \bigg( \sum_{j=n-k+1}^{n-1} \frac{1}{j} \bigg) \frac{e^{-z}}{\log n} \nonumber \\
&\leq \frac{\eps_n e^{-z}}{\log n} \bigg( \sum_{k= \lceil n/(\log n)^{1/2} + 1 \rceil}^{\lfloor n/2 \rfloor + 1} \frac{n}{k(k-1)} \cdot \frac{2(k-1)}{n} + \sum_{k = \lfloor n/2 \rfloor + 2}^{n} \frac{n}{k(k-1)} \bigg( \sum_{j=n-k+1}^{n-1} \frac{1}{j} \bigg) \bigg) \nonumber \\
&\leq \frac{\eps_n e^{-z}}{\log n} \bigg( 2 \bigg( \log n - \log \bigg( \frac{n}{(\log n)^{1/2}} \bigg) \bigg) + \frac{4}{n} \sum_{j=1}^n \sum_{k=n-j+1}^{n} \frac{1}{j} \bigg) \nonumber \\
&= \frac{\eps_n e^{-z} (\log \log n + 4)}{\log n}.
\end{align}

To bound $P(A_3)$, note that an event in which the tree is cut at some edge adjacent to the root corresponds to a merger in the Bolthausen-Sznitman coalescent that involves the block containing the integer $1$.  By the sampling consistency of the Bolthausen-Sznitman coalescent, every other block merges with this block at rate $\lambda_2 = 1$.  Consequently, the expected number of blocks that are removed from the tree between times $t$ and $t + \eps_n$ when the tree is cut at an edge adjacent to the root is $\eps_n (n-1)$.  Provided that fewer than $n/(\log n)^{1/2}$ vertices are removed as a result of the cut, the probability that a given vertex reattaches to the root is at most $2/n$.  The probability that the label on the new edge exceeds $\log \log n + z$ is $e^{-z}/\log n$ as before.  Thus, for sufficiently large $n$
\begin{equation}\label{PA3}
P(A_3 \cap A_2^c) \leq \eps_n (n-1) \cdot \frac{2}{n} \cdot \frac{e^{-z}}{\log n} \leq \frac{2e^{-z} \eps_n}{\log n}.
\end{equation}

To bound $P(A_4)$, note that there are two ways that $A_4$ can occur.  Either the tree can be cut twice between times $t$ and $t + \eps_n$, or two or more edges can reattach to the root after a single cut.  Because cuts of the tree happen at times of a Poisson process of rate $\lambda_n = n-1$, the probability that two or more cuts happen between times $t$ and $t + \eps_n$ is at most $\eps_n^2 (n-1)^2$.  If there is an event in which $k-1$ vertices are removed from the tree following a cut, there are $\binom{k-1}{2}$ pairs of vertices that could reattach to the root.  On $A_2^c$, for sufficiently large $n$, the chance that two given vertices reattach to the root is at most $4/n^2$, and each new edge independently has probability $e^{-z}/(\log n)$ of having a label greater than $\log \log n + z$.  Thus,
\begin{align}\label{PA4}
P(A_4 \cap A_2^c) &\leq \eps_n^2 (n-1)^2 + \eps_n \sum_{k=3}^{\lceil n/(\log n)^{1/2} \rceil} \frac{n}{k(k-1)} \binom{k-1}{2} \frac{4 e^{-2z}}{n^2 (\log n)^2} \nonumber \\
&\leq \eps_n^2 (n-1)^2 + \frac{2 e^{-2z} \eps_n}{(\log n)^{5/2}}.
\end{align}
The result follows from (\ref{PA2}), (\ref{PA3}), and (\ref{PA4}).
\end{proof}

\begin{Lemma}\label{A1lem}
We have $$P(A_1) \leq 2 e^{-z} \eps_n$$ for sufficiently large $n$.  Furthermore,
$$\lim_{n \rightarrow  \infty} E \bigg[ \bigg| \frac{e^z}{\eps_n} P(A_1|{\cal F}_n(t)) - 1 \bigg| \bigg] = 0.$$
\end{Lemma}

\begin{proof}
Recall that in our construction using trees, we labeled the vertices by the time in which they were added rather than using the vertex labels $1, \dots, n$.  However, for the purposes of this proof, we will arbitrarily number the vertices at time $t$ by the integers $1, \dots, n$, with the root being vertex $1$.  For $k \geq 2$, let $d_k$ be the depth of vertex $k$, which is the number of edges on the path from the root to $k$.  Let $v_k$ be the number of edges along the path from the root to $k$ that are not adjacent to the root but that have at least $n/(\log n)^{1/2}$ vertices below them.  That is, $v_k$ is the number of edges $e$ along this path not adjacent to the root such that if we cut the tree at the edge $e$, it would be classified as a large cut.  Define $D_n(t) = d_2 + \dots + d_n$ and $V_n(t) = v_2 + \dots + v_n$.  For each $k = 2, \dots, n$, we will separately bound the probability that $A_1$ occurs and that $k$ is the vertex that reattaches to the root with a label of at least $\log \log n + z$.

Note that there are $d_k - 1 - v_k$ edges not adjacent to the root such that, if the tree were cut at that edge, the vertex labeled $k$ would be removed from the tree and this would be a small cut.  The probability that one of these edges is cut before time $t + \eps_n$ is $1 - e^{-\eps_n(d_k - 1 - v_k)}$, which is between $\eps_n(d_k - 1 - v_k) - \eps_n^2 (d_k - 1 - v_k)^2/2$ and $\eps_n(d_k - 1 - v_k)$.  The probability that the vertex labeled $k$ reattaches to the root is between $1/n$ and $1/(n - n/\sqrt{\log n})$, and the probability that the new edge label is at least $\log \log n + z$ is $e^{-z}/\log n$.  Therefore,
\begin{align}\label{A1Fupper}
P(A_1|{\cal F}_n(t)) &\leq \sum_{k=2}^n \eps_n (d_k - 1 - v_k) \cdot \frac{1}{n - n/\sqrt{\log n}} \cdot \frac{e^{-z}}{\log n} \nonumber \\
&= \frac{\eps_n(D_n(t) - (n-1) - V_n(t)) e^{-z}}{n \log n} \bigg( \frac{\sqrt{\log n}}{\sqrt{\log n} - 1} \bigg) \nonumber \\
&\leq \frac{\eps_n D_n(t) e^{-z}}{n \log n} \bigg( \frac{\sqrt{\log n}}{\sqrt{\log n} - 1} \bigg).
\end{align}
Because $E[D_n(t)] \leq n \log n$ by Lemma \ref{DnLem}, it follows that $P(A_1) \leq 2 e^{-z} \eps_n$ for all $n$ large enough that $\sqrt{\log n}/(\sqrt{\log n} - 1) \leq 2$. 

Suppose the vertex $k$ is cut from the tree.  The probability that vertex $k$ and some other vertex both reattach to the root with new edge labels greater than $\log \log n + z$ is at most $$(n-2) \bigg( \frac{1}{n - n/\sqrt{\log n}} \bigg)^2 \bigg( \frac{e^{-z}}{\log n} \bigg)^2 \leq \frac{2 e^{-2z}}{n (\log n)^2}$$ for sufficiently large $n$.  Also, the probability that there are two cuts to the tree before time $t + \eps_n$ is at most $(n-1)^2\eps_n^2$.  Combining these observations, we get
\begin{align}\label{A1Flower}
P(&A_1|{\cal F}_n(t)) \nonumber \\
&\geq \sum_{k=2}^n \bigg(\bigg( \eps_n (d_k - 1 - v_k) - \frac{\eps_n^2 (d_k - 1 - v_k)^2}{2} \bigg)\bigg(\frac{e^{-z}}{n \log n} - \frac{2 e^{-2z}}{n (\log n)^2} \bigg) - \eps_n^2(n-1)^2\bigg) \nonumber \\
&\geq \frac{\eps_n(D_n(t) - (n-1) - V_n(t)) e^{-z}}{n \log n} - \frac{\eps_n^2 D_n(t)^2 e^{-z}}{2n \log n} - \frac{2 \eps_n D_n(t) e^{-2z}}{n (\log n)^2} - \eps_n^2 n^3. 
\end{align}
Combining (\ref{A1Fupper}) and (\ref{A1Flower}) gives
\begin{align}\label{6tms}
\bigg| \frac{e^z}{\eps_n} P(A_1|{\cal F}_n(t)) - 1 \bigg| &\leq \bigg| \frac{D_n(t)}{n \log n} - 1 \bigg| + \frac{(n - 1) + V_n(t)}{n \log n} \nonumber \\
&\hspace{.2in}+ \frac{D_n(t)}{n \log n} \bigg( \frac{1}{\sqrt{\log n} - 1} \bigg) + \frac{\eps_n D_n(t)^2}{2n \log n} + \frac{2 D_n(t) e^{-z}}{n (\log n)^2} + \frac{\eps_n e^z}{n^3}.
\end{align}

We need to show that the six terms on the right-hand side of (\ref{6tms}) tend to zero in expectation as $n \rightarrow \infty$.  By Lemma \ref{DnLem} and the Cauchy-Schwarz Inequality,
\begin{align}
E \bigg[ \bigg| \frac{D_n(t)}{n \log n} - 1 \bigg| \bigg] &\leq \frac{1}{n \log n} E \big[ \big| D_n(t) - E[D_n(t)] \big| \big] + \bigg| \frac{E[D_n(t)]}{n \log n} - 1 \bigg| \nonumber \\
&\leq \frac{\mbox{Var}(D_n(t))^{1/2}}{n \log n} + \bigg| \frac{E[D_n(t)]}{n \log n} - 1 \bigg| \rightarrow 0
\end{align}
as $n \rightarrow \infty$.  Because $E[V_n(t)]$ is at most $n$ times the rate of transitions in the Bolthausen-Sznitman coalescent that cause at least $n/(\log n)^{1/2}$ blocks to be lost, we have
$$E[V_n(t)] \leq n \sum_{k = \lceil n/(\log n)^{1/2} + 1 \rceil}^n \binom{n}{k} \lambda_{n,k} = n \sum_{k = \lceil n/(\log n)^{1/2} + 1 \rceil}^n \frac{n}{k(k-1)} \leq n \sqrt{\log n},$$ from which it follows that the expected value of the second term on the right-hand side of (\ref{6tms}) tends to zero as $n \rightarrow \infty$.  Using Lemma \ref{DnLem} and the fact that $\eps_n = n^{-4}$, it is easily checked that the expectations of the last four terms on the right-hand side of (\ref{6tms}) tend to zero as $n \rightarrow \infty$.
\end{proof}

\begin{Prop}\label{fddprop}
The finite-dimensional distributions of the processes $(R_n(t), t \geq 0)$ converge as $n \rightarrow \infty$ to the finite-dimensional distributions of $(R(t), t \geq 0)$.
\end{Prop}

\begin{proof}
Recall that it remains only to show (\ref{cond4}), which is equivalent to showing that
$$\lim_{n \rightarrow \infty} E \bigg[ \bigg| \frac{1}{\eps_n} E[f(R_n(t + \eps_n)) - f(R_n(t))|{\cal F}_n(t)] + f'(R_n(t)) - \int_{R_n(t)}^{\infty} e^{-y}(f(y) - R_n(t)) \: dy \bigg| \bigg] = 0,$$
where $f$ is a function in ${\cal C}$ which is constant on $[-\infty, z]$.  We evaluate $f(R_n(t + \eps_n)) - f(R_n(t))$ on the three disjoint events $A_1$, $A_2 \cup A_3 \cup A_4$, and $A_5$.

By Lemma \ref{badevents},
\begin{equation}\label{badterm}
E \bigg[ \bigg| \frac{1}{\eps_n} E[(f(R_n(t + \eps_n)) - f(R_n(t))) {\bf 1}_{A_2 \cup A_3 \cup A_4} | {\cal F}_n(t)] \bigg| \bigg] \leq \frac{2 \|f\|}{\eps_n} P(A_2 \cup A_3 \cup A_4) \rightarrow 0
\end{equation}
as $n \rightarrow \infty$.  Note that $f(R_n(t + \eps_n)) = f(R_n(t) - \eps_n)$ on $A_5$ because either the process $(R_n(t), t \geq 0)$ does not jump between times $t$ and $t + \eps_n$, in which case $R_n(t + \eps_n) = R_n(t) - \eps_n$, or else $R_n(t + \eps_n) \leq z$, in which case $f(R_n(t + \eps_n)) = f(R_n(t) - \eps_n) = f(z)$.  Therefore, using Lemmas \ref{badevents} and \ref{A1lem},
\begin{align}\label{driftterm}
E \bigg[ &\bigg| \frac{1}{\eps_n} E[(f(R_n(t + \eps_n)) - f(R_n(t))){\bf 1}_{A_5} | {\cal F}_n(t)] + f'(R_n(t)) \bigg| \bigg] \nonumber \\
&= E \bigg[ \bigg| E \bigg[ \bigg( \frac{f(R_n(t) - \eps_n) - f(R_n(t))}{\eps_n} + f'(R_n(t)) \bigg) {\bf 1}_{A_5} + f'(R_n(t)){\bf 1}_{A_5^c} \bigg| {\cal F}_n(t) \bigg] \bigg| \bigg] \nonumber \\
&\leq E \bigg[ \bigg| \frac{f(R_n(t) - \eps_n) - f(R_n(t))}{\eps_n} + f'(R_n(t)) \bigg| \bigg] + \|f'\| P(A_5^c) \nonumber \\
&= E \bigg[ \bigg| \int_{R_n(t) - \eps_n}^{R_n(t)} \frac{f'(s) - f'(R_n(t))}{\eps_n} \: ds \bigg| \bigg] + \|f'\|P(A_1 \cup A_2 \cup A_3 \cup A_4) \nonumber \\
&\leq \eps_n \|f''\| + \|f'\|(P(A_1) + P(A_2) + P(A_3) + P(A_4)) \rightarrow 0
\end{align}
as $n \rightarrow \infty$.

On $A_1$, there is a unique time $\tau \in [t, t + \eps_n]$ such that at time $\tau$, a new edge attaches to the root and is assigned a label greater than $\log \log n + z$.  Denote by $K$ the value of this label minus $\log \log n$, and let $J = \max\{K, R_n(t)\}$.  Conditional on $A_1$ and ${\cal F}_n(t)$, the distribution of $K$ has a density given by $k(y) = e^{z-y} {\bf 1}_{\{y > z\}}$.  Therefore,
\begin{align}
E[&(f(R_n(t + \eps_n)) - f(R_n(t)) {\bf 1}_{A_1}|{\cal F}_n(t)] \nonumber \\
&= E[(f(R_n(t + \eps_n)) - f(J)){\bf 1}_{A_1}|{\cal F}_n(t)] + E[(f(J) - f(R_n(t)) {\bf 1}_{A_1}|{\cal F}_n(t)] \nonumber \\
&= E[(f(R_n(t + \eps_n)) - f(J)){\bf 1}_{A_1}|{\cal F}_n(t)] + \bigg(\int_{R_n(t)}^{\infty} e^{z-y}(f(y) - f(R_n(t)) \: dy\bigg)P(A_1|{\cal F}_n(t)). \nonumber
\end{align}
Note that $J - \eps_n \leq R_n(t + \eps_n) \leq J$ on $A_1$, so $$|E[f(R_n(t + \eps_n)) - f(J)) {\bf 1}_{A_1}|{\cal F}_n(t)]| \leq \eps_n \|f'\| P(A_1|{\cal F}_n(t)).$$  It follows that
\begin{align}\label{jumpterm}
E \bigg[ \bigg| \frac{1}{\eps_n} E[&f(R_n(t + \eps_n)) - f(R_n(t)) {\bf 1}_{A_1}|{\cal F}_n(t)] - \int_{R_n(t)}^{\infty} e^{-y}(f(y) - f(R_n(t)) \: dy \bigg| \bigg] \nonumber \\
&\leq \|f'\| P(A_1) + E \bigg[ \bigg| \bigg(\int_{R_n(t)}^{\infty} e^{-y}(f(y) - f(R_n(t)) \: dy \bigg) \bigg( \frac{e^z P(A_1|{\cal F}_n(t))}{\eps_n} - 1\bigg) \bigg| \bigg] \nonumber \\
&\leq \|f'\| P(A_1) + 2 \|f\| e^{-z} E \bigg[ \bigg| \frac{e^z}{\eps_n}P(A_1|{\cal F}_n(t)) - 1 \bigg| \bigg].
\end{align}
It now follows from Lemma \ref{A1lem} that the right-hand side of (\ref{jumpterm}) tends to zero as $n \rightarrow \infty$.  The result now follows by combining this observation with (\ref{badterm}) and (\ref{driftterm}).
\end{proof}

\subsection{Tightness}

Here we show that the sequence of processes $(R_n)_{n=1}^{\infty}$ is relatively compact.  By Theorem 7.8 in Chapter 3 of \cite{ek86}, this result in combination with Proposition \ref{fddprop} implies that the processes $(R_n(t), t \geq 0)$ converge in the Skorohod topology to $(R(t), t \geq 0)$. 

For $\delta > 0$ and $t > 0$, let
$$w_n(\delta, t) = \inf_{\{t_i\}} \max_i \sup_{r,s \in [t_{i-1}, t_i)} |R_n(r) - R_n(s)|,$$ where the infimum is taken over all partitions of the form $0 = t_0 < t_1 < \dots < t_m = t$ such that $t_i - t_{i-1} > \delta$ for $i = 1, \dots, m$.  By Corollary 7.4 in Chapter 3 of \cite{ek86}, the sequence $(R_n)_{n=1}^{\infty}$ is relatively compact provided that for all $\eps > 0$ and $t > 0$, there exists $\delta > 0$ such that
\begin{equation}\label{wcond}
\limsup_{n \rightarrow \infty} P(w_n(\delta, t) \geq \eps) \leq \eps.
\end{equation}
Therefore, we can conclude the proof of Theorem \ref{MRCAthm} by showing (\ref{wcond}).

Fix $t > 0$.  Choose $\eps \in (0,1)$, and choose $y < 0$ such that $P(R_n(0) \leq y+t) < \eps/6$.  Note that if $R_n(0) > y + t$, then $R_n(s) > y$ for all $s \leq t$.  Let $$\delta = \frac{e^{2y} \eps}{36 \max\{t, 1\}}.$$  Let $0 < \tau_{1,n} < \tau_{2,n} < \dots$ denote the jump times of $(R_n(s), s \geq 0)$.  As long as $\tau_{j,n} - \tau_{j-1,n} > \delta$ for all $j$ such that $\tau_{j,n} \leq t$ and there are no jump times in $[0, \delta]$ or $[t-\delta, t]$, it is easy to choose the times $t_0, \dots, t_n$ such that $\delta < t_i - t_{i-1} < 2 \delta$ for $i = 1, \dots, n$ and for all $j$ such that $\tau_{j,n} \leq t$, we have $\tau_{j,n} = t_i$ for some $i$.  That is, there is one of the $t_i$ at every jump time of the process.  In this case, whenever $r,s \in [t_{i-1}, t_i)$, we have $|R_n(r) - R_n(s)| = |r - s| \leq 2 \delta < \eps$.

By Lemmas \ref{badevents} and \ref{A1lem} with $\delta$ in place of $\eps_n$ and $y$ in place of $z$, we have
\begin{align}\label{endjump}
\limsup_{n \rightarrow \infty} P \big( R_n(s) \neq R_n(s-&) \mbox{ for some }s \in [0, \delta] \cup [t - \delta, t] \big) \nonumber \\
&\leq \limsup_{n \rightarrow \infty} P(R_n(0) \leq y + t) + 2 e^{-y} \delta + 2e^{-y} \delta < \frac{\eps}{3}.
\end{align}
Note that if $\tau_{j,n} - \tau_{j-1,n} \leq \delta$ for some $j$ such that $\tau_{j,n} \leq t$, then there exists a nonnegative integer $k \leq t/\delta - 1$ such that $k\delta \leq \tau_{j-1,n} < \tau_{j,n} \leq \min\{t, (k+2)\delta\}$.  For two jumps of the process $(R_n(s), s \geq 0)$ to fall within the interval $[k \delta, \min\{t, (k+2)\delta\}]$, one of the following four events must occur:
\begin{itemize}
\item We have $R_n(s) \leq y$ for some $s \in [0,t]$.

\item Between times $k\delta$ and $(k+2)\delta$, there is a large cut, and one of the new edges attaches to the root and is assigned a label greater than $\log \log n + y$.

\item Between times $k\delta$ and $(k+2)\delta$, more than $3 \delta n \log n$ vertices are removed from the tree during small cuts.

\item Of the first $\lfloor 3 \delta n \log n \rfloor$ vertices, after time $k \delta$, that are removed from the tree during small cuts, two or more reattach to the root with new edge labels greater than $\log \log n + y$.
\end{itemize}

We have already bounded the probability of the first event by $\eps/6$.  The other three events depend on $k$.  The probability of the second event tends to zero as $n \rightarrow \infty$ by (\ref{PA2}) with $2 \delta$ in place of $\eps_n$ and $y$ in place of $z$.  To bound the probability of the third event, note that mergers in the Bolthausen-Sznitman coalescent in which $k-1$ blocks are lost occur at rate $\binom{n}{k} \lambda_{n,k} = n/(k(k-1))$.  Therefore, if $N$ denotes the number of vertices removed during small cuts between times $k \delta$ and $(k+2) \delta$, we have 
$$E[N] = 2 \delta \sum_{k=2}^{\lceil n/(\log n)^{1/2} \rceil} (k-1) \cdot \frac{n}{k(k-1)} = 2 \delta \sum_{k=2}^{\lceil n/(\log n)^{1/2} \rceil} \frac{n}{k} \leq 2 \delta n \log n$$ and
$$\mbox{Var}(N) = 2 \delta \sum_{k=2}^{\lceil n/(\log n)^{1/2} \rceil} (k-1)^2 \cdot \frac{n}{k(k-1)} \leq 2 \delta n^2.$$  Therefore, by Chebyshev's Inequality, $$P(N > 3 \delta n \log n) \leq P(|N - E[N]| > \delta n \log n) \leq \frac{2 \delta n^2}{\delta^2 n^2 (\log n)^2} \rightarrow 0$$ as $n \rightarrow \infty$.  Finally, concerning the fourth event, note that when a vertex is reattached after being removed during a small cut, the probability that it reattaches to the root is at most $2/n$, and the probability that it is assigned a label greater than $\log \log n + y$ is $e^{-y}/(\log n)$.  Therefore, since there are at most $(3 \delta n \log n)^2/2$ pairs of vertices to consider, the probability of the fourth event for a particular $k$ is at most $$\frac{(3 \delta n \log n)^2}{2} \cdot \frac{4 e^{-2y}}{(n \log n)^2} \leq 18 \delta^2 e^{-2y}.$$  Since there are at most $t/\delta$ possible values of $k$ to consider, the probability that the fourth event occurs for some $k$ is at most $18t \delta e^{-2 y} \leq \eps/2$.  Therefore,
$$\limsup_{n \rightarrow \infty} P(\tau_{j,n} - \tau_{j-1,n} \leq \delta \mbox{ for some }j\mbox{ such that }\tau_{j,n} \leq t) \leq \frac{\eps}{6} + \frac{\eps}{2} \leq \frac{2 \eps}{3}.$$  Combining this result with (\ref{endjump}) gives (\ref{wcond}) and completes the proof of Theorem \ref{MRCAthm}.

\section{Proof of Theorem \ref{numblocksthm}}

We obtain Theorems \ref{lengththm} and \ref{numblocksthm} using a very different approach.  Rather than using recursive trees, we couple the population model with a family of stable processes by constructing both from a Poisson process.  We describe this construction in section \ref{constsec}.  We then prove Theorem \ref{numblocksthm} in section \ref{numsec}, and we prove Theorem \ref{lengththm} in section \ref{lengthsec}.

Throughout the rest of the paper, $T > 0$ will be an arbitrary positive constant, and $$T_n = 2 \log \log n.$$  Also, $\rightarrow_p$ will denote convergence in probability as $n \rightarrow \infty$.

\subsection{A Poisson process construction}\label{constsec}

Fix a positive integer $n$.  Let $\Psi$ be a Poisson point process on $\R \times (0, \infty)$ whose intensity measure is given by $dt \times y^{-2} \: dy$.  Then define $\Theta$ to be the image of $\Psi$ under the map $(t, y) \mapsto (-t/\log n, y/\log n)$, restricted to $\R \times (0,1]$.  That is, if $(t, y)$ is a point of $\Psi$ with $y \leq \log n$, then $(-t/\log n, y/\log n)$ is a point of $\Psi$.  Note that $\Theta$ is a Poisson point process on $\R \times (0, 1]$ with intensity measure $dt \times y^{-2} \: dy$.

We now construct a population model consisting of $n$ individuals labelled $1, \dots, n$.  We independently attach to each point $(t_i, y_i)$ of $\Theta$ independent random variables $U_{i,1}, \dots, U_{i,n}$, each having the uniform distribution on $(0,1)$.  If zero or one of the random variables $U_{i,1}, \dots, U_{i,n}$ is less than $y_i$, then there is no change in the population at time $t_i$.  However, if $k \geq 2$ of the random variables $U_{i,1}, \dots, U_{i,n}$ are less than $y_i$ and the $k$ smallest of these random variables are $U_{i, j_1} < \dots < U_{i, j_k}$, then at time $t_i$, the individuals labeled $j_2, \dots, j_k$ are killed, and the individual labeled $j_1$ gives birth to $k-1$ new offspring, which assume the labels $j_2, \dots, j_k$.

To see that this is equivalent to the population model described in the introduction, note that if $(t_i, y_i)$ is a point of $\Theta$, the probability that exactly $k$ new offspring are born at time $t_i$ is $\binom{n}{k+1} y_i^{k+1} (1 - y_i)^{n-k-1}$.  Thus, the rate of events in which exactly $k$ new offspring are born is $$\binom{n}{k+1} \int_0^1 y^{k+1} (1-y)^{n-k-1} \cdot y^{-2} \: dy = \frac{n}{k(k+1)},$$ which matches (\ref{popmodel}) because changes in the population occur at rate $n-1$.

For $s \in \R$ and $t \geq 0$, let $N_n(s, t)$ denote the number of individuals in the population immediately before time $s - t$ who have a descendant alive in the population at time $s$.  That is, $N_n(s,t)$ is the number of ancestral lines remaining after time $t$ if we trace back the ancestral lines of the individuals in the population at time $s$.  Note that because we consider the population immediately before time $s-t$ rather than exactly at time $s-t$ when defining $N_n(s,t)$, the process $(N_n(s,t), t \geq 0)$ is right continuous.  Also, note that $N(s,0) = n$ as long as there is no change in the population at time $s$, but if $k$ individuals are killed at time $s$ and replaced by new offspring, then $N(s,0) = n - k$.  Let $N_n(t) = N_n(0,t)$.  Because the genealogy of this population is given by the Bolthausen-Sznitman coalescent started with $n$ blocks, the process $(N_n(t), t \geq 0)$ has the same law as the process defined in the statement of Theorem \ref{numblocksthm}.  Let
\begin{equation}\label{Lnkappa}
L_n(s) = \int_0^{\infty} N_n(s,t) {\bf 1}_{\{N_n(s,t) > 1\}} \: dt.
\end{equation}
Then $L_n(s)$ is the total branch length for the coalescent tree representing the genealogy of the population at time $s$, so $(L_n(s), s \geq 0)$ has the same law as the process considered in Theorem \ref{lengththm}.

Next, we use the Poisson process $\Psi$ to construct, for each $s \in \R$, a stable process $(S(s,t), t \geq 0)$ with characteristic exponent
\begin{align}\label{stablechf}
E[e^{iuS(s,t)}] &= \exp \bigg( - \frac{\pi t}{2} |u| + i t u \log |u| \bigg) \nonumber \\
&= \exp \bigg( itu(1 - \gamma) - t\int_{-\infty}^0 \big(1 - e^{iux} + iux {\bf 1}_{\{|x| \leq 1\}} \big) x^{-2} \: dx \bigg),
\end{align}
where $\gamma$ denotes Euler's constant.
Because $\Psi$ has only countably many points, we can enumerate the points of $\Psi$ as $(s_j, x_j)_{j=1}^{\infty}$.  If $s_j > -s$, then the process $(S(s,t), t \geq 0)$ will have a jump of size $-x_j$ at time $s_j + s$.  To make this construction precise, we use a standard approximation procedure that is described, for example, in Section I.1 of \cite{bertoin}.  Define $$\eps_n = e^{-\sqrt{\log n}},$$ which implies that for all $a > 0$, $$\lim_{n \rightarrow \infty} (\log n)^a \eps_n = 0$$ and $$\lim_{n \rightarrow \infty} n^a \eps_n = \infty.$$  Let $R_n(s, t)$ be the set of all $j$ such that $-s < s_j \leq -s + t$ and $x_j > \eps_n$.  For all $s \in \R$ and $t \geq 0$, let
\begin{equation}\label{Sndef}
S_n(s, t) = t(1 - \gamma - \log \eps_n) - \sum_{j \in R_n(s,t)} x_j.
\end{equation}
Note that $\int_{-\infty}^0 x {\bf 1}_{\{\eps_n < |x| \leq 1\}} \: x^{-2} \: dx = \log \eps_n$.  Therefore, for any fixed $T > 0$ and any integers $m,n > N$, the proof of Theorem 1 in Section I.1 of \cite{bertoin} (see the bottom of p. 14) gives
\begin{equation}\label{Esup}
E \bigg[ \sup_{0 \leq t \leq T} (S_n(s,t) - S_m(s,t))^2 \bigg] \leq 4T \int_0^{\infty} x^2 {\bf 1}_{\{|x| \leq \eps_N\}} \: x^{-2} \: dx \leq 4 T \eps_N.
\end{equation}
It follows that for each $s \in \R$, there is a limit process $(S(s,t), t \geq 0)$ satisfying (\ref{stablechf}) such that for each fixed $T > 0$,
\begin{equation}\label{Esup2}
E \bigg[ \sup_{0 \leq t \leq T} (S_n(s,t) - S(s,t))^2 \bigg] \leq 4T\eps_n \rightarrow 0
\end{equation}
as $n \rightarrow \infty$.  Let $S_n(t) = S_n(0,t)$ and $S(t) = S(0,t)$ for all $t \geq 0$.

\subsection{Bounds on stable processes}

The three lemmas below collect some bounds on these stable processes that will be needed later.

\begin{Lemma}\label{Slem}
Fix $\eps > 0$.  Then there exists $K > 0$ such that
\begin{equation}\label{etS}
P \bigg( \sup_{0 \leq s \leq T} \sup_{t \geq 0} \: e^{-t} |S(s,t)| > K \bigg) < \eps.
\end{equation}
Additionally, we have
\begin{equation}\label{supS}
\lim_{n \rightarrow \infty} P \bigg( \sup_{0 \leq s \leq T} \sup_{0 \leq t \leq T_n + s} |S(s,t)| > 5 (\log \log n)^2 \bigg) = 0.
\end{equation}
Also,
\begin{equation}\label{STn}
\lim_{n \rightarrow \infty} P(|S(T_n)| \leq T_n^2) = 1
\end{equation}
and
\begin{equation}\label{Stail}
\sup_{0 \leq s \leq T} \int_{T_n+s}^{\infty} e^{-t}S(s,t) \: dt \rightarrow_p 0.
\end{equation}
\end{Lemma}

\begin{proof}
The process $(S(t), t \geq 0)$ is a stable process of index $1$.  Therefore, by Proposition 48.10 of \cite{sato}, we have $$\limsup_{t \rightarrow \infty} t^{-\alpha}|S(t)| = 0$$ almost surely for all $\alpha > 1$.  Therefore, there exists a random time $U$ with $P(U < \infty) = 1$ such that $|S(t)| \leq t^2$ for all $t \geq U$.  This implies (\ref{STn}).

Let $L = \sup_{0 \leq t \leq U} |S(t)|$ and $M = \sup_{0 \leq t \leq T} |S(T,t)|$.  Note that $L < \infty$ and $M < \infty$ almost surely.  Suppose $0 \leq s \leq T$.  By the construction, if $0 \leq t \leq s$ then
\begin{equation}\label{Sst1}
S(s,t) = S(T, T - s + t) - S(T, T-s),
\end{equation}
while if $t \geq s$ then
\begin{equation}\label{Sst2}
S(s,t) = S(s,s) + S(0, t-s) = S(T,T) - S(T,T-s) + S(0, t-s).
\end{equation}
Therefore,
\begin{equation}\label{Sstbd}
|S(s,t)| \leq 2M + L + t^2.
\end{equation}
Since $e^{-t} t^2 \leq 1$ for all $t \geq 0$, it follows that $$\sup_{0 \leq s \leq T} \sup_{t \geq 0} \: e^{-t} |S(s,t)| \leq 2M + L + 1.$$  Now choose $K$ large enough that $P(2M + L + 1 > K) < \eps$ to obtain (\ref{etS}).
Furthermore, (\ref{Sstbd}) implies that
$$\sup_{0 \leq s \leq T} \sup_{0 \leq t \leq T_n + s} |S(s,t)| \leq 2M + L + (2 \log \log n + T)^2,$$ and (\ref{supS}) follows because $P(2M + L + (2 \log \log n + T)^2 > 5 (\log \log n)^2) \rightarrow 0$ as $n \rightarrow \infty$.  On the event that $2M + L \leq T_n^2$, which has probability tending to one as $n \rightarrow \infty$, equation (\ref{Sstbd}) gives that for $0 \leq s \leq T$,
$$\int_{T_n+s}^{\infty} e^{-t}|S(s,t)| \: dt \leq 2 \int_{T_n}^{\infty} e^{-t} t^2 \: dt \rightarrow 0$$ as $n \rightarrow \infty$, which gives (\ref{Stail}).
\end{proof}

\begin{Lemma}\label{supSnS}
We have $$\sup_{0 \leq s \leq T} \sup_{0 \leq t \leq T_n + s} |S_n(s,t) - S(s,t)| \rightarrow_p 0.$$
\end{Lemma}

\begin{proof}
Suppose $0 \leq s \leq T$.  By (\ref{Sst1}), which holds also with $S_n$ in place of $S$, if $0 \leq t \leq s$ then
\begin{align}
|S_n(s,t) - S(s,t)| &= |S_n(T, T - s + t) - S_n(T, T-s) - S(T, T-s+t) + S(T, T-s)| \nonumber \\
&\leq 2 \sup_{0 \leq u \leq T} |S_n(T,u) - S(T,u)|. \nonumber
\end{align}
By (\ref{Sst2}), if $s \leq t \leq T_n + s$, then
\begin{align}
|S_n&(s,t) - S(s,t)| \nonumber \\
&= |S_n(T,T) - S_n(T, T-s) + S_n(0, t-s) - S(T,T) + S(T,T-s) - S(0, t-s)| \nonumber \\
&\leq 2 \sup_{0 \leq u \leq T} |S_n(T,u) - S(T,u)| + \sup_{0 \leq u \leq T_n} |S_n(0,u) - S(0,u)|. \nonumber
\end{align}
Using that $(a+b)^2 \leq 2a^2 + 2b^2$  and applying (\ref{Esup2}), we get
$$E \bigg[ \bigg( \sup_{0 \leq s \leq T} \sup_{0 \leq t \leq T_n + s} |S_n(s,t) - S(s,t)| \bigg)^2 \bigg] \leq 32 T \eps_n + 8 T_n \eps_n \rightarrow 0$$ as $n \rightarrow \infty$, which implies the result.
\end{proof}

\begin{Lemma}\label{Sntheta}
Suppose $s \in \R$ and $(\delta_n)_{n=1}^{\infty}$ is a sequence of numbers with $\delta_n \sqrt{\log n} \rightarrow 0$ as $n \rightarrow \infty$.  Let $\theta > 0$, and let $A(\theta)$ be the event that there are no points of $\Psi$ in $[-s, -s + \delta_n] \times [\theta, \infty)$.  Then for sufficiently large $n$, $$P \bigg( \sup_{0 \leq t \leq \delta_n} |S_n(s, t)| \geq \theta^{1/3} \bigg| A(\theta) \bigg) \leq 4 \delta_n \theta^{1/3}.$$
\end{Lemma}

\begin{proof}
By (\ref{Sndef}), we have $$\sup_{0 \leq t \leq \delta_n} |S(s,t)| \leq \delta_n(1 - \gamma - \log \eps_n) + \sum_{j \in R_n(s,t)} x_j.$$  Because $|\log \eps_n| = \sqrt{\log n}$ and $\delta_n \sqrt{\log n} \rightarrow 0$, we have $\delta_n (1 - \gamma - \log \eps_n) \leq \theta^{1/3}/4$ for sufficiently large $n$.  Furthermore, $$E \bigg[ \sum_{j \in R_n(s,t)} x_j \bigg| A(\theta) \bigg] = \delta_n \int_{\eps_n}^{\theta} x \cdot x^{-2} \: dx = \delta_n (\log \theta - \log \eps_n) \rightarrow 0$$ as $n \rightarrow \infty$ and $$\mbox{Var} \bigg( \sum_{j \in R_n(s,t)} x_j \bigg| A(\theta) \bigg) = \delta_n \int_{\eps_n}^{\theta} x^2 \cdot x^{-2} \: dx \leq \delta_n \theta.$$  Therefore, by Chebyshev's Inequality, for sufficiently large $n$, 
\begin{align}
P \bigg( \sup_{0 \leq t \leq \delta_n} |S_n(s, t)| \geq \theta^{1/3} \bigg| A(\theta) \bigg) &\leq P \bigg( \sum_{j \in R_n(s,t)} x_j \geq \frac{3 \theta^{1/3}}{4} \bigg| A(\theta) \bigg) \nonumber \\
&\leq P \bigg( \bigg| \sum_{j \in R_n(s,t)} x_j - E \bigg[ \sum_{j \in R_n(s,t)} x_j \bigg| A(\theta) \bigg] \bigg| \geq \frac{\theta^{1/3}}{2} \bigg| A(\theta) \bigg) \nonumber \\
&\leq \delta_n \theta \bigg( \frac{4}{\theta^{2/3}} \bigg),
\end{align}
which gives the result.
\end{proof}

\subsection{Rate of decrease in the number of blocks}

We record here some results about the rate at which the number of blocks decreases in the Bolthausen-Sznitman coalescent.  Recall that the process $(N_n(t), t \geq 0)$ tracks the evolution of the number of blocks over time.  Because the number of blocks decreases by $k-1$ whenever $k$ blocks merge into one, the rate at which the number of blocks is decreasing when there are $b$ blocks is
\begin{equation}\label{etab}
\eta(b) = \sum_{k=2}^b (k-1) \binom{b}{k} \lambda_{b,k}.
\end{equation}
Considering the process from the perspective of the construction in subsection \ref{constsec}, suppose there is a point $(-t,y)$ in the Poisson process $\Theta$.  If $N_n(t-) = b$, then $N_n(t-) - N_n(t)$ is one less than the number out of $b$ independent uniformly distributed random variables that are less than or equal to $y$, unless all of the random variables are greater than $y$ in which case $N_n(t-) - N_n(t) = 0$.  Therefore, the expected decrease in the process $N_n$ at time $t$ is $by - 1 + (1-y)^b$.  It follows that $$\eta(b) = \int_0^1 \big( by - 1 + (1-y)^b \big) y^{-2} \: dy.$$

We will be interested also in the process obtained by removing from $\Theta$ the points whose second coordinate exceeds $\eps_n/\log n$.  In this case, the genealogy of the population is given not by the Bolthausen-Sznitman coalescent but by a different coalescent process in which the largest merger events are suppressed.  In particular, when there are $b$ blocks, the rate at which $k$ blocks merge into one is given by
\begin{equation}\label{lambprime}
\lambda_{b,k}^* = \int_0^{\eps_n/\log n} x^{k-2} (1-x)^{b-k} \: dx
\end{equation}
and the rate at which the number of blocks is decreasing is given by $$\eta^*(b) = \sum_{k=2}^b (k-1) \binom{b}{k} \lambda_{b,k}^* = \int_0^{\eps_n/\log n} \big( by - 1 + (1-y)^b \big) y^{-2} \: dy.$$  Also, let
\begin{equation}\label{vstar}
v^*(b) = \sum_{k=2}^b (k-1)^2 \binom{b}{k} \lambda_{b,k}^*.
\end{equation}

\begin{Lemma}
For all positive integers $b$ and $n$ with $2 \leq b \leq n$, we have
\begin{equation}\label{meancomp}
\big| \eta^*(b) - b(\log b - \log \log n + \log \eps_n + \gamma - 1) \big| \leq \frac{\log n}{\eps_n} + 1,
\end{equation}
where $\gamma$ denotes Euler's constant.  Also,
\begin{equation}\label{vcomp}
v^*(b) \leq \frac{b^2 \eps_n}{\log n}.
\end{equation}
\end{Lemma}

\begin{proof}
Using (\ref{etab}) and (\ref{lambform}), we get
$$\eta(b) = b \sum_{k=2}^b \frac{1}{k}.$$
Now $$\lim_{b \rightarrow \infty} \bigg( \sum_{k=2}^b \frac{1}{k} - \log b \bigg) = \gamma - 1$$ and $$0 \leq \big( \log b - \log (b-1) \big) - \frac{1}{b} = \int_{b-1}^b \bigg( \frac{1}{x} - \frac{1}{b} \bigg) \: dx \leq \frac{1}{b-1} - \frac{1}{b},$$
so
\begin{equation}\label{etab2}
\big| \eta(b) - b(\log b + \gamma - 1) \big| \leq b \sum_{k=b+1}^{\infty} \bigg( \frac{1}{k-1} - \frac{1}{k} \bigg) = 1
\end{equation}
for all $b$.  It follows that
\begin{align}\label{etaprimeb}
\eta^*(b) &= \int_0^{\eps_n/\log n} \big( by - 1 + (1-y)^b \big) y^{-2} \: dy. \nonumber \\
&= \int_0^1 \big( by - 1 + (1-y)^b \big) y^{-2} \: dy - \int_{\eps_n/\log n}^1 \big( by - 1 + (1-y)^b \big) y^{-2} \: dy \nonumber \\
&= \eta(b) + b \log \bigg( \frac{\eps_n}{\log n} \bigg) + \int_{\eps_n/\log n}^1  \big(1 - (1-y)^b \big) y^{-2} \: dy.
\end{align}
Because the last integral is bounded by $\int_{\eps_n/\log n}^{\infty} y^{-2} \: dy = (\log n)/\eps_n$, the result (\ref{meancomp}) follows by combining (\ref{etab2}) and (\ref{etaprimeb}).

To prove (\ref{vcomp}), note that if $(-t, y)$ is a point of $\Theta$ and $N_n(t-) = b$, then because $(k-1)^2 \leq 2 \binom{k}{2}$ for $k \geq 2$, the expected square of the decrease in the process $N_n$ at time $t$ is at most twice the expected number of pairs out of $b$ independent uniformly distributed random variables having the property that both random variables are less than or equal to $y$, which is $2 \binom{b}{2} y^2$.  Therefore, $$v^*(b) \leq \int_0^{\eps_n/\log n} 2 \binom{b}{2} y^2 \cdot y^{-2} \: dy \leq \frac{b^2 \eps_n}{\log n},$$ as claimed.
\end{proof}

\subsection{The coupling}\label{numsec}

In this section, we prove Theorem \ref{numblocksthm}.  We use the construction and notation of section \ref{constsec}.  As in the statement of Theorem \ref{numblocksthm}, let
\begin{equation}\label{Xndef}
X_n(t) = \frac{\log n}{n} \bigg( N_n \bigg( \frac{t}{\log n} \bigg) - ne^{-t} - \frac{nte^{-t} \log \log n}{\log n} \bigg)
\end{equation}
for all $t \geq 0$.  Also, let $$Y_n(t) = e^{-t} S_n(t) + \frac{e^{-t} t^2}{2}$$ and $$Y(t) = e^{-t} S(t) + \frac{e^{-t} t^2}{2}.$$  We need to show that the processes $(X_n(t), t \geq 0)$ converge to $(Y(t), t \geq 0)$ in the Skorohod topology.  Because $$E \bigg[ \sup_{0 \leq t \leq T_n} (Y_n(t) - Y(t))^2 \bigg] \leq 4T_n \eps_n \rightarrow 0$$ as $n \rightarrow \infty$ by (\ref{Esup2}), it suffices to show that
\begin{equation}\label{XYprob}
\sup_{0 \leq t \leq T_n} |X_n(t) - Y_n(t)| \rightarrow_p 0.
\end{equation}
In fact, to prove Theorem \ref{numblocksthm}, it would suffice to show (\ref{XYprob}) with the arbitrary fixed constant $T$ in place of $T_n$, but it will be helpful for the proof of Theorem \ref{lengththm} to control the difference between $X_n$ and $Y_n$ up to time $T_n$.

Let $0 < \tau_1 < \dots < \tau_{J_n} < T_n$ be the jump times of the process $(S_n(t), t \geq 0)$ before time $T_n$.  Let $\tau_0 = 0$ and $\tau_{J_n + 1} = T_n$.  Note that the $\tau_i$ depend on $n$ even though we do not record this dependence in the notation.  This means that there are points $(\tau_i, y_i)$ in $\Psi$ with $y_i \geq \eps_n$ for $i = 1, \dots, J_n$.  Also, the process $\Theta$, which is used to construct the population process, contains the points $(-\tau_i/\log n, y_i/\log n)$ but $\Theta$ contains no points in the regions $(-\tau_{i+1}/\log n, -\tau_i/\log n) \times [\eps_n/\log n, 1]$.  Therefore, conditional on $\tau_1, \dots, \tau_{J_n}$, between times $\tau_i/\log n$ and $\tau_{i+1}/\log n$, the process $N_n$ follows the dynamics of the number of blocks in a coalescent process with transition rates given by (\ref{lambprime}).

For $i = 0, 1, \dots, J_n$ and $t \in [0, \tau_{i+1} - \tau_i)$, let
\begin{equation}\label{Mindef}
M_{i,n}(t) = N_n \bigg( \frac{\tau_i + t}{\log n} \bigg) - N_n \bigg( \frac{\tau_i}{\log n} \bigg) + \int_{\tau_i/\log n}^{(\tau_i+t)/\log n} \eta^*(N_n(s)) \: ds.
\end{equation}
By standard results about compensators for Markov jump processes (see, for example, Theorem 9.15 in \cite{kle}), the process $(M_{i,n}(t), 0 \leq t < \tau_{i+1} - \tau_i)$ is a martingale.  Note also that $\tau_{i+1} - \tau_i$ is exponentially distributed with mean $\eps_n$ and is independent of the evolution of the process $M_{i,n}$ before time $\tau_{i+1} - \tau_i$.  Next, for $\tau_j \leq t < \tau_{j+1}$, let
\begin{equation}\label{Mndef}
M_n(t) = \sum_{i=0}^{j-1} M_{i,n}((\tau_{i+1} - \tau_i)-) + M_{j,n}(t - \tau_j).
\end{equation}
Then the process $(M_n(t), 0 \leq t < T_n)$ is a martingale.

\begin{Lemma}\label{martlem}
We have $$P \bigg( \frac{\log n}{n} \sup_{0 \leq t \leq T_n} |M_n(t)| > \frac{1}{\log n} \bigg) \leq 4 (\log n)^2 T_n \eps_n.$$ 
\end{Lemma}

\begin{proof}
By standard results about compensated Markov jump processes (see, for example, Corollary 9.17 of \cite{kle}),
$$\mbox{Var}(M_n(T_n)) = E \bigg[ \int_0^{T_n/\log n} v^*(N_n(t)) \: dt \bigg],$$ where $v^*$ was defined in (\ref{vstar}). Therefore, by (\ref{vcomp}),
$$\mbox{Var}(M_n(T_n)) \leq \frac{n^2 T_n \eps_n}{(\log n)^2}.$$  By the $L^2$ Maximum Inequality for martingales,
$$E  \bigg[ \sup_{0 \leq t \leq T_n} M_n(t)^2 \bigg] \leq \frac{4 n^2 T_n \eps_n}{(\log n)^2}.$$  Thus, by Markov's Inequality, $$P \bigg( \frac{\log n}{n} \sup_{0 \leq t \leq T_n} |M_n(t)| > \frac{1}{\log n} \bigg) = P \bigg( \sup_{0 \leq t \leq T_n} M_n(t)^2 > \frac{n^2}{(\log n)^4} \bigg) \leq 4 (\log n)^2 T_n \eps_n,$$ as claimed.
\end{proof}

For $0 \leq t \leq T_n$, let
\begin{equation}\label{Zndef}
Z_n(t) = X_n(t) - \frac{\log n}{n} M_n(t) = \frac{\log n}{n} \bigg( N_n \bigg( \frac{t}{\log n} \bigg) - M_n(t) - ne^{-t} - \frac{n t e^{-t} \log \log n}{\log n} \bigg).
\end{equation}
Lemma \ref{martlem} implies that $$\sup_{0 \leq t \leq T_n} \frac{\log n}{n} \big|M_n(t)| \rightarrow_p 0$$ as $n \rightarrow \infty$.  Therefore, to show (\ref{XYprob}), it suffices to compare the processes $Z_n$ and $Y_n$.

For $s \geq 0$, let $(\tau_1^s, y_1^s), (\tau_2^s, y_2^s), \dots, (\tau_{J_n^s}^s, y_n^s)$ denote the points of $\Psi$ in $[-s, T_n] \times [\eps_n, \infty)$, ranked so that $\tau_1^s < \dots < \tau_{J_n^s}^s$.  Let $\tau_0^s = -s$ and $\tau_{J_n^s+1}^s = T_n$.  Note that $\tau_i^0 = \tau_i$ for $i = 0, 1, \dots, J_n + 1$.  Also let $y_i = y_i^0$ for $i = 1, \dots, J_n$.  In the next Lemma, we show that several events hold with high probability.  This result will allow us to assume that these events holds throughout much of the rest of the paper.

\begin{Lemma}\label{Anlem}
Let $A_{1,n}$ be the event that $J_n^T \leq n^{1/4}$.
Let $A_{2,n}$ be the event that $$\sum_{i=1}^{J_n^T} y_i^T \leq (\log n)^{3/4}.$$  Let $A_{3,n}$ be the event that $\tau_{i+1}^T - \tau_i^T \leq \eps_n \log n$ for $i = 0, 1, \dots, J_n^T$.  Let $A_n = A_{1,n} \cap A_{2,n} \cap A_{3,n}$.  Then $\lim_{n \rightarrow \infty} P(A_n) = 1$.  
\end{Lemma}

\begin{proof}
Note that $$E[J_n^T] = (T_n + T) \int_{\eps_n}^{\infty} y^{-2} \: dy = \frac{(T_n + T)}{\eps_n}.$$  It follows from Markov's Inequality that $$\limsup_{n \rightarrow \infty} P(A_{1,n}^c) \leq \limsup_{n \rightarrow \infty} \frac{(T_n + T)}{n^{1/4} \eps_n} = 0,$$ and thus $\lim_{n \rightarrow \infty} P(A_{1,n}) = 1$.

We have $$P \bigg( \max_{1 \leq i \leq J_n^T} y_i > \log n \bigg) \leq (T_n + T) \int_{\log n}^{\infty} y^{-2} \: dy = \frac{T_n + T}{\log n} \rightarrow 0$$ as $n \rightarrow \infty$.
Also, $$E \bigg[ \sum_{i=1}^{J_n^T} y_i^T \bigg| \max_{1 \leq i \leq J_n^T} y_i^T \leq \log n \bigg] = (T_n + T) \int_{\eps_n}^{\log n} y \cdot y^{-2} \: dy = (T_n + T) (\log \log n - \log \eps_n).$$  Because $-\log \eps_n = (\log n)^{1/2}$ and thus $(T_n + T)(\log \log n - \log \eps_n)/(\log n)^{3/4} \rightarrow 0$ as $n \rightarrow \infty$, it now follows from Markov's Inequality that $\lim_{n \rightarrow \infty} P(A_{2,n}) = 1$. 

To estimate $P(A_{3,n})$, let $a_k = -T + (k \eps_n \log n)/2$ for $k = 0, 1, \dots$, and let $K = \min\{k: a_k > T_n\}$.  Note that if $\tau_{i+1}^T - \tau_i^T > \eps_n \log n$, then some interval of the form $[a_{k-1}, a_k]$ with $1 \leq k \leq K$ must not contain any of the points $\tau_i^T$.  The probability that $[a_{k-1}, a_k]$ does not contain any of the $\tau_i^T$ is $$\exp \bigg( - (a_k - a_{k-1}) \int_{\eps_n}^{\infty} y^{-2} \: dy \bigg) = \exp \bigg( - \frac{\eps_n \log n}{2} \cdot \frac{1}{\eps_n} \bigg) = n^{-1/2}.$$  It follows that for sufficiently large $n$, $$P(A_{3,n}^c) \leq K n^{-1/2} \leq \bigg( \frac{2(T_n + T)}{\eps_n \log n} + 1 \bigg) n^{-1/2} \rightarrow 0$$ as $n \rightarrow \infty$.  Therefore, $\lim_{n \rightarrow \infty} P(A_{3,n}) = 1$, which completes the proof.
\end{proof}

The next lemma shows that the processes $Y_n$ and $Z_n$ typically jump by approximately the same amount at the jump times $\tau_i$.

\begin{Lemma}\label{couplejumps}
We have
\begin{align}
&P \bigg( A_n \cap \bigg\{ \sum_{i=1}^{J_n} \big| (Z_n(\tau_i) - Z_n(\tau_i-)) - (Y_n(\tau_i) - Y_n(\tau_i-)) \big| {\bf 1}_{\{|X_n(\tau_i-)| \leq \log \log n\}} > \frac{1}{(\log n)^{1/8}} \bigg\} \bigg) \nonumber \\
&\hspace{5.2in}\leq \frac{T_n (\log n)^{3/2}}{n^{1/2} \eps_n} \nonumber
\end{align}
for sufficiently large $n$.
\end{Lemma}

\begin{proof}
Recall that $(\tau_i, y_i)$ is a point of $\Psi$ for $i = 1, \dots, J_n$.  Therefore, $S_n(\tau_i) - S_n(\tau_i-) = -y_i$, so
\begin{equation}\label{yjump}
Y_n(\tau_i) - Y_n(\tau_i-) = -e^{-\tau_i} y_i.
\end{equation}
Let $\tau_i = T_n$ and $y_i = 0$ on $\{i > J_n\}$, and let ${\cal G}_i = \sigma(\tau_i, N_n((\tau_i/\log n)-), y_i)$.  Then on $\{i \leq J_n\}$,
\begin{equation}\label{NB}
N_n \bigg( \frac{\tau_i}{\log n} \bigg) - N_n \bigg( \frac{\tau_i}{\log n} - \bigg) = \min\{0, 1 - B_i\},
\end{equation}
where, conditional on ${\cal G}_i$, the distribution of $B_i$ is binomial with parameters $N_n((\tau_i/\log n)-)$ and $y_i/\log n$.  Here $B_i$ represents the number of lineages out of $N_n((\tau_i/\log n)-)$ that merge at time $\tau_i/\log n$.  We have $\mbox{Var}(B_i|{\cal G}_i) \leq n y_i/\log n$ on $\{i \leq J_n\}$, so by the Cauchy-Schwarz Inequality, $$E \bigg[ \bigg| B_i - N_n \bigg( \frac{\tau_i}{\log n}- \bigg) \frac{y_i}{\log n} \bigg| \big| {\cal G}_i \bigg] \leq E \bigg[ \bigg( B_i - N_n \bigg( \frac{\tau_i}{\log n}- \bigg) \frac{y_i}{\log n} \bigg)^2 \bigg| {\cal G}_i \bigg]^{1/2} \leq \bigg( \frac{n y_i}{\log n} \bigg)^{1/2}$$ on $\{i \leq J_n\}$.  Multiplying both sides by ${\bf 1}_{\{i \leq J_n\}} {\bf 1}_{\{y_i \leq \log n\}}$, which is ${\cal G}_i$-measurable, and taking expectations gives
$$E \bigg[ \bigg| B_i - N_n \bigg( \frac{\tau_i}{\log n}- \bigg) \frac{y_i}{\log n} \bigg| {\bf 1}_{\{i \leq J_n\}} {\bf 1}_{\{y_i \leq \log n\}} \bigg] \leq n^{1/2} P(i \leq J_n).$$  Summing over $i$ and observing that $A_{2,n} \cap \{i \leq J_n\} \subset \{y_i \leq \log n\}$ for all $i$, we get
$$E \bigg[ {\bf 1}_{A_{2,n}} \sum_{i=1}^{J_n} \bigg| B_i - N_n \bigg( \frac{\tau_i}{\log n}- \bigg) \frac{y_i}{\log n} \bigg| \bigg]  \leq n^{1/2} E[J_n] = n^{1/2} T_n \int_{\eps_n}^{\infty} y^{-2} \: dy = \frac{n^{1/2} T_n}{\eps_n}.$$  Thus,
\begin{equation}\label{ZY1}
P \bigg( A_{2,n} \cap \bigg\{ \frac{\log n}{n} \sum_{i=1}^{J_n} \bigg| B_i - N_n \bigg( \frac{\tau_i}{\log n}- \bigg) \frac{y_i}{\log n} \bigg| > \frac{1}{(\log n)^{1/2}} \bigg\}\bigg) \leq \frac{T_n (\log n)^{3/2}}{n^{1/2} \eps_n}.
\end{equation}
Also, 
\begin{align}
\bigg| \frac{y_i}{n} N_n \bigg( \frac{\tau_i}{\log n} - \bigg) - e^{-\tau_i} y_i \bigg| &= \bigg| \frac{y_i}{n} \bigg( \frac{n}{\log n} X_n(\tau_i-) + ne^{-\tau_i} + \frac{n \tau_i e^{-\tau_i} \log \log n}{\log n} \bigg) - e^{-\tau_i} y_i \bigg| \nonumber \\
&= y_i \bigg| \frac{X_n(\tau_i-)}{\log n} + \frac{\tau_i e^{-\tau_i} \log \log n}{\log n} \bigg|, \nonumber
\end{align}
which on the event $\{|X_n(\tau_i-)| \leq \log \log n\}$ is bounded by $2 y_i (\log \log n)/\log n$.  Thus, on $A_{2,n}$, 
\begin{equation}\label{ZY2}
\sum_{i=1}^{J_n} \bigg| \frac{y_i}{n} N_n \bigg( \frac{\tau_i}{\log n}- \bigg) - e^{-\tau_i} y_i \bigg| {\bf 1}_{\{|X_n(\tau_i-)| \leq \log \log n\}} \leq \frac{2 \log \log n}{\log n} \sum_{i=1}^{J_n} y_i \leq \frac{2 \log \log n}{(\log n)^{1/4}}.
\end{equation}
By (\ref{NB}), $$\bigg| Z_n(\tau_i) - Z_n(\tau_i-) + \frac{\log n}{n} B_i \bigg| = \bigg| \frac{\log n}{n} \bigg( N_n \bigg( \frac{\tau_i}{\log n} \bigg) - N_n \bigg( \frac{\tau_i}{\log n} - \bigg) + B_i \bigg) \bigg| \leq \frac{\log n}{n},$$ and so on $A_{1,n}$,
\begin{equation}\label{ZY3}
\sum_{i=1}^{J_n} \bigg| Z_n(\tau_i) - Z_n(\tau_i-) + \frac{\log n}{n} B_i \bigg| \leq \frac{\log n}{n^{3/4}}.
\end{equation}
Combining (\ref{ZY1}), (\ref{ZY2}), and (\ref{ZY3}), we get
\begin{align}
&P \bigg(A_n \cap \bigg\{ \sum_{i=1}^{J_n} \big|(Z_n(\tau_i) - Z_n(\tau_i-)) + e^{-\tau_i} y_i \big| {\bf 1}_{\{|X_n(\tau_i-)| \leq \log \log n\}}\nonumber \\
&\hspace{1.5in} > \frac{1}{(\log n)^{1/2}} + \frac{2 \log \log n}{(\log n)^{1/4}} + \frac{\log n}{n^{3/4}} \bigg\} \bigg) \leq \frac{T_n (\log n)^{3/2}}{n^{1/2} \eps_n}.
\end{align}
The result follows by combining this result with (\ref{yjump}) and using that
$$\frac{1}{(\log n)^{1/2}} + \frac{2 \log \log n}{(\log n)^{1/4}} + \frac{\log n}{n^{3/4}} \leq \frac{1}{(\log n)^{1/8}}$$ for sufficiently large $n$.
\end{proof}

Lemmas \ref{Ydrift} and \ref{Zdrift} below pertain to the behavior of the processes $Y_n$ and $Z_n$ in between the jump times $\tau_i$.

\begin{Lemma} \label{Ydrift}
If $0 \leq h < \tau_{i+1} - \tau_i \leq \eps_n \log n$ and $n$ is sufficiently large, then on $\{i \leq J_n\}$, we have $$\bigg|Y_n(\tau_i + h) - Y_n(\tau_i) - h e^{-\tau_i}(1 - \gamma - \log \eps_n + \tau_i) + \int_{\tau_i}^{\tau_i + h} Y_n(s) \: ds \bigg| {\bf 1}_{\{S_n(\tau_i) \leq \frac{1}{4} \log n\}} \leq h^2 \log n.$$
\end{Lemma}

\begin{proof}
By the construction of the process $S_n$, we have $S_n(\tau_i + h) - S_n(\tau_i) = h(1 - \gamma - \log \eps_n)$.  Using $O(h^2)$ to denote an expression whose absolute value is at most $h^2$, Taylor's Theorem gives $e^{-(\tau_i + h)} = e^{-\tau_i} - he^{-\tau_i} + O(h^2)$ and $\frac{1}{2} e^{-(\tau_i + h)} (\tau_i + h)^2 = \frac{1}{2} e^{-\tau_i} \tau_i^2 + h(\tau_i e^{-\tau_i} - \frac{1}{2} \tau_i^2 e^{-\tau_i}) + O(h^2)$.  Therefore,
\begin{align}
Y_n(\tau_i + h) - Y_n(\tau_i) &= e^{-(\tau_i + h)} S_n(\tau_i + h) - e^{-\tau_i} S_n(\tau_i) + \frac{e^{-(\tau_i + h)}(\tau_i + h)^2}{2} - \frac{e^{-\tau_i} \tau_i^2}{2} \nonumber \\
&= (e^{-\tau_i} - h e^{-\tau_i} + O(h^2)) (S_n(\tau_i) + h(1 - \gamma - \log \eps_n)) - e^{-\tau_i} S_n(\tau_i) \nonumber \\
&\hspace{2.3in}+ h\bigg(\tau_i e^{-\tau_i} - \frac{1}{2} \tau_i^2 e^{-\tau_i} \bigg) + O(h^2) \nonumber \\
&= h e^{-\tau_i}(1 - \gamma - \log \eps_n + \tau_i) - h \bigg(e^{-\tau_i} S_n(\tau_i) + \frac{e^{-\tau_i} \tau_i^2}{2} \bigg) \nonumber \\
&\hspace{.2in} + O(h^2) - h^2 e^{-\tau_i} (1 - \gamma - \log \eps_n) + O(h^2)(S_n(\tau_i) + h(1 - \gamma - \log \eps_n)). \nonumber
\end{align}
On the event that $S_n(\tau_i) \leq \frac{1}{4} \log n$, for sufficiently large $n$, the absolute value of the sum of the first two terms is bounded above by $\frac{1}{2} h \log n$, while the absolute value of the sum of the last three terms is bounded above by $\frac{1}{2} h^2 \log n$.  Therefore, for sufficiently large $n$, on the event that $S_n(\tau_i) \leq \frac{1}{4} \log n$, we have
\begin{equation}\label{h21}
\bigg|Y_n(\tau_i + h) - Y_n(\tau_i) - h e^{-\tau_i}(1 - \gamma - \log \eps_n + \tau_i) + hY_n(\tau_i) \bigg| \leq \frac{1}{2} h^2 \log n.
\end{equation}
Also, for sufficiently large $n$, if $0 \leq s \leq h$, then $|Y_n(\tau_i + s) - Y_n(\tau_i)| \leq \frac{1}{2} s \log n + \frac{1}{2}s^2 \log n \leq s \log n$ on the event that $S_n(\tau_i) \leq \frac{1}{4} \log n$ because $h \leq \eps_n \log n \leq 1$.  Therefore,
\begin{equation}\label{h22}
\bigg| \int_{\tau_i}^{\tau_i + h} Y_n(s) \: ds - hY_n(\tau_i) \bigg| \leq \int_0^h |Y_n(\tau_i + s) - Y_n(\tau_i)| \: ds \leq \int_0^h s \log n \: ds = \frac{1}{2} h^2 \log n.
\end{equation}
By combining (\ref{h21}) and (\ref{h22}), we arrive at the statement of the lemma.
\end{proof}

\begin{Lemma}\label{Nnbound}
Suppose $0 \leq h < \tau_{i+1} - \tau_i \leq \eps_n \log n$ and $\tau_i \leq s \leq \tau_i + h$.  If $n$ is sufficiently large and $|X_n(s)| \leq \log \log n$, then
\begin{equation}\label{claimNn}
N_n \bigg( \frac{s}{\log n} \bigg) \bigg| \log N_n \bigg( \frac{s}{\log n} \bigg) - \log(n e^{-\tau_i}) \bigg| \leq \frac{8 n (\log \log n)^2}{\log n}.
\end{equation}
\end{Lemma}

\begin{proof}
In view of (\ref{Xndef}), if $\tau_i \leq s \leq \tau_i + h$, then, using that $h \leq \eps_n \log n \leq (\log \log n)/\log n$ for sufficiently large $n$ and using the assumption that $|X_n(s)| \leq \log \log n$, we get
\begin{align} \label{Nndiff}
\bigg| N_n \bigg( \frac{s}{\log n} \bigg) - ne^{-\tau_i} \bigg| &= \bigg| \frac{n X_n(s)}{\log n} + ne^{-s} + \frac{nse^{-s} \log \log n}{\log n} - ne^{-\tau_i} \bigg| \nonumber \\
&\leq \big|ne^{-s} - ne^{-\tau_i} \big| + \frac{n s e^{-s} (\log \log n)}{\log n} + \bigg| \frac{n X_n(s)}{\log n} \bigg| \nonumber \\
&\leq nh + \frac{n \log \log n}{\log n} + \frac{n \log \log n}{\log n} \nonumber \\
&\leq \frac{3n \log \log n}{\log n}
\end{align}
for sufficiently large $n$.

We consider two cases.  First, suppose $n e^{-\tau_i} \geq (4n \log \log n)/\log n$.  Because (\ref{Nndiff}) implies that $ne^{-\tau_i} \geq \frac{1}{4} N_n(s/\log n)$, we get, using (\ref{Nndiff}) and the fact that $\frac{d}{dx} \log x = 1/x$, $$\bigg| \log N_n \bigg( \frac{s}{\log n} \bigg) - \log (ne^{-\tau_i}) \bigg| \leq \bigg( \frac{3n \log \log n}{\log n} \bigg) \bigg( \frac{1}{4} N_n \bigg( \frac{s}{\log n} \bigg) \bigg)^{-1}.$$  Therefore, $$N_n \bigg( \frac{s}{\log n} \bigg) \bigg| \log N_n \bigg( \frac{s}{\log n} \bigg) - \log(n e^{-\tau_i}) \bigg| \leq \frac{12n \log \log n}{\log n}.$$

Next, suppose that $ne^{-\tau_i} \leq (4n \log \log n)/\log n$.  To determine the value of $N_n(s/(\log n))$ that maximizes the left-hand side of (\ref{claimNn}), we consider the function $f(x) = x(\log x - \log a)$, where $a > 0$ is a constant.  Note that $f'(x) = 1 + \log (x/a)$, which is positive when $x > a/e$ and negative when $x < a/e$.  Thus, $f(x)$ is negative when $x < a$ and reaches its minimum when $x = a/e$, and $f(x)$ is positive and increasing for $x > a$.  We conclude that (\ref{claimNn}) will hold in general provided that it holds if $ne^{-\tau_i}/e$ or $ne^{-\tau_i} + (3n \log \log n)/\log n$ is plugged in for $N_n(s/(\log n))$.  Note that by (\ref{Nndiff}), we need not consider larger values.  In the former case, the expression that we get is $$\frac{ne^{-\tau_i}}{e} \bigg| \log \bigg( \frac{ne^{-\tau_i}}{e} \bigg) - \log(n e^{-\tau_i}) \bigg| = \frac{ne^{-\tau_i}}{e} \leq \frac{4n \log \log n}{e \log n}.$$  In the latter case, since $n e^{-\tau_i} \geq n e^{-T_n} = n/(\log n)^2$, the expression that we get is bounded above by
\begin{align}
\frac{7 n \log \log n}{\log n} &\bigg( \log \bigg( \frac{7n \log \log n}{\log n} \bigg) - \log \bigg(\frac{n}{(\log n)^2} \bigg) \bigg) \nonumber \\
&\leq \frac{7n \log \log n}{\log n} \bigg( \log 7 + \log \log \log n + \log \log n \bigg). \nonumber
\end{align}
Thus, (\ref{claimNn}) holds for sufficiently large $n$.
\end{proof}

\begin{Lemma}\label{Zdrift}
If $0 \leq h < \tau_{i+1} - \tau_i \leq \eps_n \log n$ and $n$ is sufficiently large, then on $\{i \leq J_n\}$, we have
\begin{align}
&\bigg|Z_n(\tau_i + h) - Z_n(\tau_i) - h e^{-\tau_i}(1 - \gamma - \log \eps_n + \tau_i) + \int_{\tau_i}^{\tau_i + h} X_n(t) \: dt \bigg| {\bf 1}_{\{|X_n(s)| \leq \log \log n \: \forall \: s \in [\tau_i, \tau_i + h]\}} \nonumber \\
&\hspace{4in}\leq h^2 (\log n)^{1/2} + \frac{3 h \log \log n}{(\log n)^{1/2}}. \nonumber
\end{align}
\end{Lemma}

\begin{proof}
Throughout the proof, we will work on the event that $i \leq J_n$ and $|X_n(s)| \leq \log \log n$ for all $s \in [\tau_i, \tau_i + h]$.
By the definition (\ref{Zndef}) of $Z_n$,
\begin{align}
Z_n(\tau_i + h) - Z_n(\tau_i) &= \frac{\log n}{n} \bigg( N_n \bigg( \frac{\tau_i + h}{\log n} \bigg) - N_n \bigg( \frac{\tau_i}{\log n} \bigg) - M_n(\tau_i + h) + M_n(\tau_i) \nonumber \\
&\hspace{.2in} - n e^{-(\tau_i + h)} + n e^{-\tau_i} - \frac{n(\tau_i + h) e^{-(\tau_i + h)} \log \log n}{\log n} + \frac{n \tau_i e^{-\tau_i} \log \log n}{\log n} \bigg). \nonumber
\end{align}
Using the definition of $M_n$ in (\ref{Mindef}) and (\ref{Mndef}),  we get
\begin{align}\label{Zninc}
Z_n(\tau_i + h) - Z_n(\tau_i) &= - \frac{\log n}{n} \int_{\tau_i/\log n}^{(\tau_i + h)/\log n} \eta^*(N_n(t)) \: dt + (\log n)(e^{-\tau_i} - e^{-(\tau_i + h)}) \nonumber \\
&\hspace{2in} + (\log \log n)(\tau_i e^{-\tau_i} - (\tau_i + h)e^{-(\tau_i + h)}).
\end{align}

To estimate the integral, we need to estimate $\eta^*(N_n(t))$.  Equation (\ref{meancomp}), which holds for sufficiently large $n$ even when $b = 1$ if we define $\eta^*(1) = 0$, gives
\begin{equation}\label{etapN}
\big| \eta^*(N_n(t)) - N_n(t)(\log N_n(t) - \log \log n + \log \eps_n + \gamma - 1) \big| \leq \frac{\log n}{\eps_n} + 1.
\end{equation}
It follows from (\ref{etapN}) and Lemma \ref{Nnbound}, after making the substitution $t = s/\log n$, that if $\tau_i/\log n \leq t \leq (\tau_i + h)/\log n$, then
\begin{equation}\label{etaapprox}
\big| \eta^*(N_n(t)) - N_n(t)(\log n - \log \log n + \log \eps_n + \gamma - 1 - \tau_i) \big| \leq \frac{9 n (\log \log n)^2}{\log n}
\end{equation}
for sufficiently large $n$.

Using (\ref{Xndef}), we have
\begin{align}
\int_{\tau_i/\log n}^{(\tau_i + h)/\log n} N_n(t) \: dt &= \frac{1}{\log n} \int_{\tau_i}^{\tau_i + h} N_n \bigg( \frac{s}{\log n} \bigg) \: ds \nonumber \\
&= \frac{1}{\log n} \int_{\tau_i}^{\tau_i + h} n e^{-s} + \frac{n s e^{-s} \log \log n}{\log n} + \frac{n X_n(s)}{\log n} \: ds \nonumber \\
&= \frac{n (e^{-\tau_i} - e^{-(\tau_i + h)})}{\log n} +  \frac{n \log \log n}{(\log n)^2} \big((\tau_i + 1)e^{-\tau_i} - (\tau_i + h + 1)e^{-(\tau_i + h)} \big) \nonumber \\
&\hspace{2.5in}+ \frac{n}{(\log n)^2} \int_{\tau_i}^{\tau_i + h} X_n(s) \: ds. \nonumber
\end{align}
Combining this result with (\ref{etaapprox}) and letting $\xi$ denote an expression whose absolute value is less than $9h(\log \log n)^2/(\log n)$, we get
\begin{align}\label{etaint}
\frac{\log n}{n} &\int_{\tau_i/\log n}^{(\tau_i + h)/\log n} \eta^*(N_n(t)) \: dt \nonumber \\
&= \frac{(\log n)(\log n - \log \log n + \log \eps_n + \gamma - 1 - \tau_i)}{n} \int_{\tau_i/\log n}^{(\tau_i + h)/\log n} N_n(t) \: dt + \xi \nonumber \\
&= (\log n - \log \log n + \log \eps_n + \gamma - 1 - \tau_i)(e^{-\tau_i} - e^{-(\tau_i + h)}) \nonumber \\
&\hspace{.2in}+ (\log \log n) \big( (\tau_i + 1)e^{-\tau_i} - (\tau_i + h + 1)e^{-(\tau_i + h)}) \big) + \int_{\tau_i}^{\tau_i + h} X_n(s) \: ds \nonumber \\
&\hspace{.2in}+ \bigg( \frac{(-\log \log n + \log \eps_n + \gamma - 1 - \tau_i)(\log \log n)}{\log n}\bigg) \big((\tau_i + 1)e^{-\tau_i} - (\tau_i + h + 1)e^{-(\tau_i + h)} \big) \nonumber \\
&\hspace{.2in}+  \bigg( \frac{-\log \log n + \log \eps_n + \gamma - 1 - \tau_i}{\log n}\bigg) \int_{\tau_i}^{\tau_i + h} X_n(s) \: ds + \xi. 
\end{align}
Recall that $\log \eps_n = -(\log n)^{1/2}$.  One can show using calculus that if $g(x) = (x+1)e^{-x}$, then $|g'(x)| \leq 1$ for all $x > 0$.  Consequently, we have $|(\tau_i + 1)e^{-\tau_i} - (\tau_i + h + 1)e^{-(\tau_i + h)})| \leq h$.  Also, because we are assuming $|X_n(s)| \leq \log \log n$ for all $s \in [\tau_i \tau_i + h]$, we have $|\int_{\tau_i}^{\tau_i + h} X_n(s) \: ds| \leq h \log \log n$.  It follows from these observations that the sum of the terms on the last two lines of (\ref{etaint}) is bounded by $3h(\log \log n)(\log n)^{-1/2}$ for sufficiently large $n$.  Therefore,
\begin{align}
\frac{\log n}{n} \int_{\tau_i/\log n}^{(\tau_i + h)/\log n} \eta^*(N_n(t)) \: dt &= (\log n)(e^{-\tau_i} - e^{-(\tau_i + h)}) + (\log \log n)(\tau_i e^{-\tau_i} - (\tau_i + h)e^{-(\tau_i + h)}) \nonumber \\
&\hspace{.2in}+ (\log \eps_n + \gamma - 1 - \tau_i)(e^{-\tau_i} - e^{-(\tau_i + h)}) + \int_{\tau_i}^{\tau_i + h} X_n(s) \: ds + \xi', \nonumber
\end{align}
where $|\xi'| \leq 3 h (\log \log n)(\log n)^{-1/2}$.  Combining this result with (\ref{Zninc}), we get
$$Z_n(\tau_i + h) - Z_n(\tau_i) = (1 - \gamma - \log \eps_n + \tau_i)(e^{-\tau_i} - e^{-(\tau_i + h)}) - \int_{\tau_i}^{\tau_i + h} X_n(s) \: ds - \xi'.$$  Because $-\log \eps_n = (\log n)^{1/2}$ and $|(e^{-\tau_i} - e^{-(\tau_i + h)}) - he^{-\tau_i}| \leq h^2/2$, the result follows.
\end{proof}

\begin{Lemma}\label{Bnlem}
Define the events
\begin{align}
B_{1,n} &= \bigg\{ \sup_{0 \leq s \leq T} \sup_{t \geq 0} e^{-t} |S(s,t)| \leq \frac{1}{2} \log \log n \bigg\}. \nonumber \\
B_{2,n} &= \bigg\{ \sup_{0 \leq s \leq T} \sup_{0 \leq t \leq T_n+s} |S(s,t)| \leq 5 (\log \log n)^2 \bigg\}. \nonumber \\
B_{3,n} &= \bigg\{ \sup_{0 \leq s \leq T} \sup_{0 \leq t \leq T_n+s} |S_n(s,t) - S(s,t)| \leq 1 \bigg\}. \nonumber
\end{align}
Let $B_n = B_{1,n} \cap B_{2,n} \cap B_{3,n}$.  Then $\lim_{n \rightarrow \infty} P(B_n) = 1$.
\end{Lemma}

\begin{proof}
Note that $B_{1,n}$ is the complement of the event on the right-hand side of (\ref{etS}) with $\frac{1}{2} \log \log n$ in place of $K$, and $B_{2,n}$ is the complement of the event in (\ref{supS}).  Therefore, by Lemma \ref{Slem}, we have $\lim_{n \rightarrow \infty} P(B_{1,n}) = \lim_{n \rightarrow \infty} P(B_{2,n}) = 1$.  We also have $\lim_{n \rightarrow \infty} P(B_{3,n}) = 1$ by Lemma \ref{supSnS}.
\end{proof}

Recall that the event $A_n$ was defined in Lemma \ref{Anlem}.  The next lemma shows that when $A_n$ and $B_n$ occur, it is unlikely that the processes $X_n$ and $Y_n$ are ever far apart before time $T_n$.  In view of Lemmas \ref{Anlem} and \ref{Bnlem}, this result implies (\ref{XYprob}), and therefore Theorem \ref{numblocksthm}.

\begin{Lemma}\label{mainXYprob}
We have $$P \bigg( A_n \cap B_n \cap \bigg\{ \sup_{0 \leq t \leq T_n} |X_n(t) - Y_n(t)| > \frac{4}{(\log n)^{1/16}} \bigg\} \bigg) \leq \frac{T_n (\log n)^{3/2}}{n^{1/2} \eps_n} + 4 (\log n)^2 T_n \eps_n$$ for sufficiently large $n$.
\end{Lemma}

\begin{proof}
We claim that on the event
\begin{align}\label{bigevent}
A_n &\cap B_n \cap \bigg\{ \frac{\log n}{n} \sup_{0 \leq t \leq T_n} |M_n(t)| \leq \frac{1}{\log n} \bigg\} \nonumber \\
&\cap \bigg\{ \sum_{i=1}^{J_n} \big| (Z_n(\tau_i) - Z_n(\tau_i-)) - (Y_n(\tau_i) - Y_n(\tau_i-)) \big| {\bf 1}_{\{|X_n(\tau_i-)| \leq \log \log n\}} \leq \frac{1}{(\log n)^{1/8}} \bigg\},
\end{align}
we have $|Z_n(t) - Y_n(t)| \leq 3/(\log n)^{1/16}$ for all $t \leq T_n$, and therefore $|X_n(t) - Y_n(t)| \leq 4/(\log n)^{1/16}$ for all $t \leq T_n$ by (\ref{Zndef}).  Thus, by Lemmas \ref{martlem} and \ref{couplejumps}, this claim implies the result.

To prove the claim, let $\zeta = \inf\{t \leq T_n: |Z_n(t) - Y_n(t)| > 3/(\log n)^{1/16}\}$.  We assume that we are working on the event in (\ref{bigevent}), and we must show that $\zeta$ is the infimum of the empty set, and thus that $\zeta = \infty$.  On $B_{1,n} \cap B_{3,n}$, we have
\begin{equation}\label{YB1n}
Y_n(t) \leq \frac{1}{2} \log \log n + \frac{e^{-t} t^2}{2} + |Y_n(t) - Y(t)| \leq \frac{1}{2} \log \log n + 2
\end{equation}
for all $t \leq T_n$.  Therefore, if $t < \zeta$ and $n$ is sufficiently large, then
\begin{align}
|X_n(t)| &\leq |Y_n(t)| + |Z_n(t) - Y_n(t)| + \frac{\log n}{n}|M_n(t)| \nonumber \\
&\leq \frac{1}{2}\log \log n + 2 + \frac{3}{(\log n)^{1/16}} + \frac{1}{\log n} \nonumber \\
&\leq \log \log n.
\end{align}
Note also that $S_n(t) \leq \frac{1}{4} \log n$ for all $t \leq T_n$ on $B_{2,n} \cap B_{3,n}$ for sufficiently large $n$, and $\tau_{i+1} - \tau_i \leq \eps_n \log n$ for $i = 0, 1, \dots, J_n$ on $A_{3,n}$.   Therefore, Lemmas \ref{Ydrift} and \ref{Zdrift} imply that if $\tau_i + h < \zeta$ and $h < \tau_{i+1} - \tau_i$, then
\begin{align}
&\bigg|(Z_n(\tau_i + h) - Z_n(\tau_i)) - (Y_n(\tau_i + h) - Y_n(\tau_i)) + \int_{\tau_i}^{\tau_i + h} (X_n(t) - Y_n(t)) \: dt \bigg| \nonumber \\
&\hspace{3in} \leq \frac{3h \log \log n}{(\log n)^{1/2}} + h^2 (\log n)^{1/2} + h^2 \log n. \nonumber
\end{align}
Since $\int_{\tau_i}^{\tau_i + h} |X_n(t) - Z_n(t)| \: dy \leq h/(\log n)$, it follows that
\begin{equation}\label{YZdrift}
\bigg|(Z_n(\tau_i + h) - Z_n(\tau_i)) - (Y_n(\tau_i + h) - Y_n(\tau_i)) + \int_{\tau_i}^{\tau_i + h} (Z_n(t) - Y_n(t)) \: dt \bigg| \leq \frac{4h \log \log n}{(\log n)^{1/2}}
\end{equation}
for sufficiently large $n$.

Define the sets
\begin{align}
I_1 &= \bigg\{i \in \{0, 1, \dots, J_n + 1\}: |Z_n(\tau_i) - Y_n(\tau_i)| \leq \frac{1}{(\log n)^{1/16}} \bigg\}, \nonumber \\
I_2 &= \bigg\{i \in \{0, 1, \dots, J_n + 1\}: Z_n(\tau_i) - Y_n(\tau_i) > \frac{1}{(\log n)^{1/16}} \bigg\}, \nonumber \\
I_3 &= \bigg\{i \in \{0, 1, \dots, J_n + 1\}: Z_n(\tau_i) - Y_n(\tau_i) < - \frac{1}{(\log n)^{1/16}} \bigg\}. \nonumber
\end{align}
Clearly, we have $I_1 \cup I_2 \cup I_3 = \{0, 1, \dots, J_n + 1\}$.  By right continuity and the fact that $Z_n(0) = Y_n(0) = 0$, we have $\zeta > 0 = \tau_0$.  Suppose that $\zeta \in (\tau_i, \tau_{i+1})$ for some $i \in I_1$.  Then, by (\ref{YZdrift}) and the fact that $\tau_{i+1} - \tau_i \leq \eps_n \log n$,
\begin{align}\label{I1bd}
|Z_n(\zeta) - Y_n(\zeta)| &\leq |Z_n(\tau_i) - Y_n(\tau_i)| + \int_{\tau_i}^{\zeta} |Z_n(t) - Y_n(t)| \: dt + \frac{4 (\zeta - \tau_i) \log \log n}{(\log n)^{1/2}} \nonumber \\
&\leq \frac{1}{(\log n)^{1/16}} + \eps_n \log n \cdot \frac{3}{(\log n)^{1/16}} + 4 \eps_n (\log n)^{1/2} \log \log n,
\end{align}
a contradiction for sufficiently large $n$ because $|Z_n(\zeta) - Y_n(\zeta)| \geq 3/(\log n)^{1/16}$ by right continuity.  Thus, if $i \in I_1$, then $\zeta \geq \tau_{i+1}$.  Therefore, if $i \in I_1$, then reasoning as in (\ref{I1bd}) and using the fact that we are working on the event (\ref{bigevent}), we get
\begin{align}\label{I1jump}
|Z_n(\tau_{i+1}) - Y_n(\tau_{i+1})| &\leq |Z_n(\tau_i) - Y_n(\tau_i)| + \int_{\tau_i}^{\tau_{i+1}} |Z_n(t) - Y_n(t)| \: dt + \frac{4 (\tau_{i+1} - \tau_i) \log \log n}{(\log n)^{1/2}} \nonumber \\
&\hspace{.5in}+ |(Z_n(\tau_{i+1}) - Z_n(\tau_{i+1}-)) - (Y_n(\tau_{i+1}) - Y_n(\tau_{i+1}-))| \nonumber \\
&\leq \frac{1}{(\log n)^{1/16}} + \eps_n \log n \cdot \frac{3}{(\log n)^{1/16}} + 4 \eps_n (\log n)^{1/2} \log \log n + \frac{1}{(\log n)^{1/8}} \nonumber \\
&\leq \frac{2}{(\log n)^{1/16}} 
\end{align}
for sufficiently large $n$, which implies that $\zeta \neq \tau_{i+1}$.

Next, suppose that $i \in I_2$.  Let $\rho = \inf\{t > \tau_i: Z_n(t) \leq Y_n(t)\}$.  Suppose $\rho \leq \min\{\zeta, \tau_{i+1}\}$.
Then $0 \leq Z_n(s) - Y_n(s) \leq 3/(\log n)^{1/16}$ for all $s < \rho$.  Therefore by (\ref{YZdrift}),
\begin{align}
Z_n(\rho) - Y_n(\rho) &\geq Z_n(\tau_i) - Y_n(\tau_i) - \int_{\tau_i}^{\rho} (Z_n(t) - Y_n(t)) \: dt - \frac{4(\rho - \tau_i) \log \log n}{(\log n)^{1/2}} - \frac{1}{(\log n)^{1/8}} \nonumber \\
&\geq \frac{1}{(\log n)^{1/16}} - \eps_n \log n \cdot \frac{3}{(\log n)^{1/16}} - 4 \eps_n (\log n)^{1/2} \log \log n - \frac{1}{(\log n)^{1/8}},
\end{align}
which is positive for sufficiently large $n$, a contradiction.  Therefore, for sufficiently large $n$, if $i \in I_2$, then $Z_n$ must stay greater than $Y_n$ from time $\tau_i$ until after time $\min\{\zeta, \tau_{i+1}\}$.  In particular, if $i \in I_2$ and $\zeta \geq \tau_{i+1}$, then $i+1 \in I_1 \cup I_2$.  By the same argument with the roles of $Y_n$ and $Z_n$ reversed, if $i \in I_3$ and $\zeta \geq \tau_{i+1}$, then $i+1 \in I_1 \cup I_3$.  It follows from these observations and (\ref{I1jump}) that the only way we could have $\zeta \in (\tau_i, \tau_{i+1}]$ with $i \in I_2$ is if there exists $j < i$ such that $j \in I_1$, $j+1, j+2, \dots, i \in I_2$, and $$Z_n(\zeta) - Y_n(\zeta) > Z_n(\tau_{j+1}) - Y_n(\tau_{j+1}) + \frac{1}{(\log n)^{1/16}}.$$  However, if this is true, then $Z_n(t) > Y_n(t)$ for all $t \in [\tau_{j+1}, \zeta]$.  Therefore, using (\ref{YZdrift}) and the fact that we are working on the event in (\ref{bigevent}), we have
\begin{align}
Z_n(\zeta) - Y_n(\zeta) &\leq Z_n(\tau_{j+1}) - Y_n(\tau_{j+1}) + \sum_{k=j+2}^{i+1} \big| (Z_n(\tau_k) - Z_n(\tau_k-)) - (Y_n(\tau_k) - Y_n(\tau_k-)) \big| \nonumber \\
&\hspace{2in}+ \frac{4(\zeta - \tau_{j-1}) (\log \log n)}{(\log n)^{1/2}} \nonumber \\
&\leq Z_n(\tau_{j+1}) - Y_n(\tau_{j+1}) + \frac{1}{(\log n)^{1/8}} + \frac{4 T_n (\log \log n)}{(\log n)^{1/2}}. \nonumber
\end{align}
For sufficiently large $n$, the sum of the last two terms on the right-hand side is less than $1/(\log n)^{1/16}$, a contradiction.  Hence, we can not have $\zeta \in (\tau_i, \tau_{i+1}]$ for $i \in I_2$, and the same argument with the roles of $Y_n$ and $Z_n$ reversed establishes that we can not have $\zeta \in (\tau_i, \tau_{i+1}]$ with $i \in I_3$.  We conclude that $\zeta = \infty$, which completes the proof.
\end{proof}

\subsection{Extension to arbitrary starting times}

The result (\ref{XYprob}) pertains to the evolution of the number of lineages when we trace back the ancestral lines of the individuals in the population at time zero.  Of course, by stationarity, the analogous result holds if we instead trace back the ancestral lines of the individuals in the population at some other time $s \geq 0$.  However, to help with the proof of Theorem \ref{lengththm} in the next section, we will need a stronger version of this result that will make it possible to show that the approximation works well simultaneously at many times.  The key to this result is that the events $A_n$ and $B_n$ were defined so that the bounds that hold on these events are valid simultaneously for all $s \in [0,T]$.

Recall that $N_n(s,t)$ denotes the number of individuals in the population immediately before time $s-t$ with a descendant alive in the population at time $s$.  Let ${\tilde N}_n(s,t) = N_n(s,t) {\bf 1}_{\{N(s,t) > 1\}}$, which will be convenient because
\begin{equation}\label{Lntilde}
L_n(s) = \int_0^{\infty} {\tilde N}_n(s,t) \: dt.
\end{equation}
Let $$X_n(s,t) = \frac{\log n}{n} \bigg( {\tilde N}_n \bigg( \frac{s}{\log n}, \frac{t}{\log n} \bigg) - n e^{-t} - \frac{nt e^{-t} \log \log n}{\log n} \bigg).$$  Also, let $$Y_n(s,t) = e^{-t} S_n(s,t) + \frac{e^{-t} t^2}{2}.$$  Note that $X_n(0,t) = X_n(t)$ for all $t \leq (\log n) \inf\{s: N_n(s) = 1\}$ and $Y_n(0,t) = Y_n(t)$ for all $t \geq 0$.

We have the following extension of Lemma \ref{mainXYprob}.  The result follows from the same argument that gives Lemma \ref{mainXYprob}.  We have replaced $4/(\log n)^{1/16}$ by $5/(\log n)^{1/16}$ to account for the error in replacing $N_n(s,t)$ by ${\tilde N}_n(s,t)$.  Indeed, it would be possible to use $4/(\log n)^{1/16} + (\log n)/n$ because $|N_n(s,t) - {\tilde N}_n(s,t)| \leq 1$ for all $s$ and $t$.  Also, because an interval of time $T_n + s$ rather than $T_n$ must be considered when adapting the proofs of Lemmas \ref{martlem} and \ref{couplejumps}, the bounds involving $T_n$ have been replaced by bounds involving $T_n + T$.

\begin{Lemma}\label{mainXYprob2}
For sufficiently large $n$, we have
\begin{align}
P \bigg( A_n \cap B_n \cap \bigg\{ \sup_{0 \leq t \leq T_n+s} |X_n(s,t) &- Y_n(s,t)| > \frac{5}{(\log n)^{1/16}} \bigg\} \bigg) \nonumber \\
&\leq \frac{(T_n + T) (\log n)^{3/2}}{n^{1/2} \eps_n} + 4 (\log n)^2 (T_n + T) \eps_n \nonumber
\end{align}
for each fixed $s \in [0, T]$.
\end{Lemma}

\section{Proof of Theorem \ref{lengththm}}\label{lengthsec}

Recall that for each $s \in \R$, a process $(S(s,t), t \geq 0)$ was defined in section \ref{constsec}.  Let $J(s) = y$ if $(-s, y)$ is a point of $\Psi$, and let $J(s) = 0$ otherwise.  For each $s \in \R$ and $t \geq 0$, let $$Y(s,t) = e^{-t} S(s,t) + \frac{e^{-t} t^2}{2},$$ and then let $$L(s) = -J(s) + \int_0^{\infty} Y(s,t) \: dt.$$  Note that $P(J(s) = 0) = 1$ for each $s \geq 0$, but the $-J(s)$ term keeps the process $(L(s), s \geq 0)$ right continuous at the jump times and thus ensures that $(L(s), s \geq 0)$ has paths that are almost surely right continuous with left limits.  To see this, note that if $(-s, y)$ is a point of $\Psi$, then $S(s+, t) = S(s,t) - y$ and thus $Y(s+, t) = Y(s,t) - e^{-t} y$ for all $t > 0$.  Therefore, $$\int_0^{\infty} Y(s+, t) \: dt = \int_0^{\infty} Y(s,t) \: dt - y.$$
The next result shows that $(L(s), s \geq 0)$ has the same law as the process defined in the statement of Theorem \ref{lengththm}.

\begin{Prop}\label{LOUprop}
Let $\nu$ be the measure on $\R$ whose density with respect to Lebesgue measure is given by $x^{-2} {\bf 1}_{(-\infty, 0)}(x)$.  The process $(L(s), s \geq 0)$ is a process of Ornstein-Uhlenbeck type generated by $(2 - \gamma, 0, \nu, 1)$.
\end{Prop}

\begin{proof}
Fix $t \geq 0$.  Note that if $0 \leq s \leq t$ and $u \geq 0$, then $S(s,u) = S(t, t-s + u) - S(t, t-s)$.  We use the notation $\int_0^{\infty} f(x) \: dS_t(x)$ to denote the stochastic integral of $f(x)$ with respect to the stable process $(S(t, x), x \geq 0)$.  If $0 \leq s \leq t$, then
\begin{align}\label{Lseq}
L(s) &= 1 - J(s) + \int_0^{\infty} e^{-u} S(s,u) \: du \nonumber \\
&= 1 - J(s) + \int_0^{\infty} e^{-u} \bigg( \int_{t-s}^{t-s+u} dS_t(x) \bigg) du \nonumber \\ 
&= 1 - J(s) + \int_{t-s}^{\infty} \bigg( \int_{x-t+s}^{\infty} e^{-u} \: du \bigg) dS_t(x) \nonumber \\
&= 1 - J(s) + \int_{t-s}^{\infty} e^{-(x-t+s)} \: dS_t(x), 
\end{align}
where in the next-to-last step we used Fubini's Theorem for general stochastic integrals (see, for example, Theorem 45 in Part IV of \cite{protter} and the remark on p. 161 in \cite{protter} that the measure $\mu$ in that theorem can be taken to be $\sigma$-finite rather than finite).  Note that $J(0) = 0$ almost surely, and almost surely $\int_0^t J(s) \: ds = 0$ for all $t$. Therefore, using Fubini's Theorem again and (\ref{Lseq}), we get
\begin{align}
L(t) - L(0) + \int_0^t L(s) \: ds &= t - J(t) + \int_0^{\infty} e^{-x} \: dS_t(x) - \int_t^{\infty} e^{-(x-t)} \: dS_t(x) \nonumber \\
&\hspace{.6in}+ \int_0^t \int_{t-s}^{\infty} e^{-(x-t+s)} \: dS_t(x) \: ds \nonumber \\
&= t - J(t) + \int_0^{\infty} e^{-x} \: dS_t(x) - \int_t^{\infty} e^{-(x-t)} \: dS_t(x) \nonumber \\
&\hspace{.6in}+ \int_0^{\infty} \int_{\max\{0, t-x\}}^t e^{-(x-t+s)} \: ds \:dS_t(x) \nonumber \\
&= t - J(t) + \int_0^t \bigg( e^{-x} + \int_{t-x}^t e^{-(x-t+s)} \: ds \bigg) dS_t(x) \nonumber \\
&\hspace{.6in}+ \int_t^{\infty} \bigg(e^{-x} - e^{-(x-t)} + \int_0^t e^{-(x-t+s)} \: ds \bigg) dS_t(x)
\end{align}
for all $t \geq 0$ almost surely.
Because $$e^{-x} + \int_{t-x}^t e^{-(x-t+s)} \: ds = 1$$ and $$e^{-x} - e^{-(x-t)} + \int_0^t e^{-(x-t+s)} \: ds = 0,$$ it follows that $$L(t) - L(0) + \int_0^t L(s) \: ds = t - J(t) + \int_0^t dS_t(x) = t - J(t) + S(t,t)$$ for all $t \geq 0$ almost surely.  Therefore, if we define $Z(t) = S(t,t) + t - J(t)$ for all $t \geq 0$, then
\begin{equation}\label{LOU}
L(t) = L(0) + Z(t) - \int_0^t L(s) \: ds,
\end{equation}
which is exactly (\ref{OUSDE}) with $c = 1$ and $L$ in place of $X$.

By the symmetry of the construction in subsection \ref{constsec}, the processes $(S(0,t), t \geq 0)$ and $(S(t,t) - J(t), t \geq 0)$ have the same law.  Indeed, if one reflects the points of the Poisson process about the vertical axis, so that a point at $(t,x)$ is moved to $(-t, x)$, and then follows the procedure used  to construct $(S(0,t), t \geq 0)$, one obtains the process $(S(t,t) - J(t), t \geq 0)$ provided that $J(0) = 0$.  Therefore, the process $(S(t,t) - J(t), t \geq 0)$ is a stable process whose characteristic exponent is given by the expression in (\ref{stablechf}).  Since $(Z(t), t \geq 0)$ differs from $(S(t,t) - J(t), t \geq 0)$ only by the addition of a linear drift of rate $1$, it follows that $(Z(t), t \geq 0)$ is a stable process whose characteristic exponent is obtained by replacing $1 - \gamma$ with $2 - \gamma$ on the right-hand side of (\ref{stablechf}).  This observation, combined with (\ref{LOU}), implies the result.  
\end{proof}

\begin{Rmk}
{\em The process $(L(s), s \geq 0)$ is clearly stationary by construction.  Thus, it follows from Propositon \ref{LOUprop} and the theory of processes of Ornstein-Uhlenbeck type reviewed in the introduction that the distribution of $L(0)$, and therefore the distribution of $L(s)$ for any fixed $s \geq 0$, has a characteristic function given by the right-hand side of (\ref{chfX}).  To observe this result more directly, use the Integration by Parts Formula (see, for example, Corollary 8.7 of \cite{kle}) to write that almost surely
\begin{equation}\label{L1d}
L(0) = \int_0^{\infty} Y(t) \: dt = 1 + \int_0^{\infty} e^{-t} S(t) \: dt = 1 + \int_0^{\infty} e^{-t} dS(t).
\end{equation}
The distribution of the stable integral on the right-hand side of (\ref{L1d}) can be evaluated using the theory developed in Chapter 3 of \cite{sam}.  In this case, we apply Proposition 3.4.1 of \cite{sam} with $m$ being $\pi/2$ times Lebesgue measure, $E = [0, \infty)$, $\sigma = \pi/2$, $\mu = 0$, $\beta = -1$, and $f(t) = e^{-t}$.  We get $\sigma_f = \pi/2$, $\beta_f = -1$, and $\mu_f = -\int_0^{\infty} te^{-t} \: dt = -1$ to recover the result.} 
\end{Rmk}

By (\ref{Lntilde}), we have
\begin{align}
&\frac{(\log n)^2}{n} \bigg(L_n \bigg( \frac{s}{\log n} \bigg) - \frac{n}{\log n} - \frac{n \log \log n}{(\log n)^2} \bigg) \nonumber \\
&\hspace{.3in}= \frac{(\log n)^2}{n} \int_0^{\infty} \frac{1}{\log n} {\tilde N}_n\bigg(\frac{s}{\log n}, \frac{t}{\log n} \bigg) - \frac{n e^{-t}}{\log n} - \frac{n t e^{-t} \log \log n}{(\log n)^2} \: dt = \int_0^{\infty} X_n(s, t) \: dt. \nonumber
\end{align}
Consequently, the following proposition will imply Theorem \ref{lengththm}.

\begin{Prop}\label{LWcouple}
We have $$\sup_{0 \leq s \leq T} \bigg|\int_0^{\infty} X_n(s,t) \: dt - L(s)\bigg| \rightarrow_p 0.$$
\end{Prop}

To prove Proposition \ref{LWcouple}, we need to compare the integrals $\int_0^{\infty} X_n(s,t) \: dt$ and $\int_0^{\infty} Y(s,t) \: dt$.  Lemmas \ref{Ytailint} and \ref{Xtailint} below will show that it suffices to consider the integrals of $X_n$ and $Y$ up to time $T_n + s$, and Lemma \ref{YnYlem} will allow us to replace $Y$ by $Y_n$.

\begin{Lemma}\label{Ytailint}
We have $$\sup_{0 \leq s \leq T} \int_{T_n + s}^{\infty} Y(s,t) \: dt \rightarrow_p 0.$$
\end{Lemma}

\begin{proof}
We have $$\sup_{0 \leq s \leq T} \int_{T_n + s}^{\infty} e^{-t} S(s,t) \: dt \rightarrow_p 0$$ by (\ref{Stail}).  Also, it is clear that $$\lim_{n \rightarrow \infty} \int_{T_n+s}^{\infty} \frac{e^{-t} t^2}{2} \: dt = 0.$$  The result is immediate.
\end{proof}

\begin{Lemma}\label{NnTn}
For all $\eps > 0$, there exists a positive constant $C$ such that
$$\liminf_{n \rightarrow \infty} P \bigg( N_n \bigg( \frac{T_n}{\log n} \bigg) \leq \frac{Cn}{(\log n)^2} \bigg) > 1 - \eps.$$
\end{Lemma}

\begin{proof}
Let $(\Pi_n(t), t \geq 0)$ be a Bolthausen-Sznitman coalescent started with $n$ blocks.  There is a well-known method for constructing a random partition having the same distribution as $\Pi_n(t)$.  Let $\alpha = e^{-t}$.  Then let $J_1 \geq J_2 \geq \dots$ denote the points of a Poisson point process on $(0, \infty)$ whose intensity measure is given by $x^{-1-\alpha} \: dx$.  Let $S = \sum_{i=1}^{\infty} J_i$, and divide the interval $[0,1]$ into disjoint subintervals of lengths $J_1/S, J_2/S, \dots$.  Let $U_1, U_2, \dots$ be i.i.d. random variables having the uniform distribution on $[0,1]$.  Then $\Pi_n(t)$ has the same distribution as the partition of $\{1, \dots, n\}$ such that $i$ and $j$ are in the same block if and only if $U_i$ and $U_j$ land in the same subinterval of $[0,1]$; see, for example, section 4 of \cite{beleg00} or section 2.4 of \cite{pit99}.  Therefore, $N_n(t)$ has the same distribution as the number of blocks of $\Pi_n(t)$.  Thus, by (3.13) of \cite{pitcsp}, $$E[N_n(t)] = \frac{\Gamma(n + \alpha)}{\alpha \Gamma(\alpha) \Gamma(n)}.$$  Therefore, by Stirling's Formula, there exist positive constants $C_1$ and $C_2$ such that for all $n \geq 2$ and $t \geq 0$,
\begin{equation}\label{boszfdd1}
E[N_n(t)] \leq \frac{C_1 (n - 1 + \alpha)^{n + \alpha - 1/2} e^{-(n -1 + \alpha)}}{\alpha \Gamma(\alpha) (n-1)^{n - 1/2} e^{-(n-1)}} \leq \frac{C_2}{\alpha \Gamma(\alpha)} n^{\alpha}.
\end{equation}
Now, take $t = T_n/(\log n)$, so that $\alpha = e^{-T_n/(\log n)}$.  Because $$\alpha \leq 1 - \frac{T_n}{\log n} + \frac{T_n^2}{2 (\log n)^2},$$ we have
\begin{equation}\label{boszfdd2}
n^{\alpha} \leq n e^{-T_n + T_n^2/2 \log n} \leq \frac{C_3 n}{(\log n)^2}
\end{equation}
for some positive constant $C_3$.  Combining (\ref{boszfdd1}) and (\ref{boszfdd2}), we get that
$$E \bigg[ N_n \bigg( \frac{T_n}{\log n} \bigg) \bigg] \leq \frac{C_4 n}{(\log n)^2}$$ for some positive constant $C_4$.
The result now follows from Markov's Inequality.
\end{proof}

\begin{Lemma}\label{Xtailint}
We have $$\sup_{0 \leq s \leq T} \int_{T_n + s}^{\infty} X_n(s,t) \: dt \rightarrow_p 0.$$
\end{Lemma}

\begin{proof}
It is easy to check that $$\lim_{n \rightarrow \infty} \frac{\log n}{n} \int_{T_n}^{\infty} \bigg( n e^{-t} + \frac{nte^{-t} \log \log n}{\log n} \bigg) \: dt = 0.$$  Therefore, it suffices to show that $$\sup_{0 \leq s \leq T} \frac{\log n}{n} \int_{T_n + s}^{\infty} {\tilde N}_n \bigg( \frac{s}{\log n}, \frac{t}{\log n} \bigg) \: dt \rightarrow_p 0.$$  Note that $$N_n \bigg( \frac{s}{\log n}, \frac{t}{\log n} \bigg) \leq N_n \bigg(0, \frac{t - s}{\log n} \bigg)$$ because the number of individuals in the population immediately before time $(s-t)/\log n$ who have descendants in the population at time $0$ will be at least as large as the number of individuals in the population immediately before time $(s-t)/\log n$ who have descendants in the population at the later time $s/(\log n)$.  It follows that
\begin{align}\label{Ninttail}
\sup_{0 \leq s \leq T} \frac{\log n}{n} \int_{T_n + s}^{\infty} {\tilde N}_n \bigg( \frac{s}{\log n}, \frac{t}{\log n} \bigg) \: dt &\leq \sup_{0 \leq s \leq T} \frac{\log n}{n} \int_{T_n + s}^{\infty} {\tilde N}_n \bigg( 0, \frac{t-s}{\log n} \bigg) \: dt \nonumber \\
&= \frac{\log n}{n} \int_{T_n}^{\infty} {\tilde N}_n \bigg(0, \frac{t}{\log n} \bigg) \: dt \nonumber \\
&= \frac{(\log n)^2}{n} \int_{T_n/\log n}^{\infty} {\tilde N}_n(0,t) \: dt. 
\end{align}
Conditional on $N_n(0, T_n/\log n) = m \geq 2$, the distribution of $\int_{T_n/\log n}^{\infty} {\tilde N}_n(0, t) \: dt$ is the same as the distribution of $L_m(0)$.  Let $\eps > 0$.  By Lemma \ref{NnTn}, there is a positive constant $C$ such that $P(N_n(0, T_n/\log n) \leq Cn/(\log n)^2) > 1 - \eps$ for sufficiently large $n$.  Conditional on the event that $N_n(0, T_n/\log n) \leq Cn/(\log n)^2$, the distribution of $\int_{T_n/\log n}^{\infty} {\tilde N}_n(0, t) \: dt$ is stochastically dominated by the distribution of $L_m(0)$ for $m = \lfloor Cn/(\log n)^2 \rfloor$.  Thus, by (\ref{dimr}), there is a positive constant $K$ such that $$P \bigg( \int_{T_n/\log n}^{\infty} N_n(0, t) \: dt \leq \frac{K n}{(\log n)^3} \bigg) > 1 - 2\eps$$ for sufficiently large $n$.  Hence, the right-hand side of (\ref{Ninttail}) converges in probability to zero, which gives the result.
\end{proof}

\begin{Lemma}\label{YnYlem}
We have $$\sup_{0 \leq s \leq T} \int_0^{T_n + s} |Y_n(s,t) - Y(s,t)| \: dt \rightarrow_p 0.$$
\end{Lemma}

\begin{proof}
Note that
\begin{align}
\int_0^{T_n + s} |Y_n(s,t) - Y(s,t)| \: dt &= \int_0^{T_n + s} e^{-t} |S_n(s,t) - S(s,t)| \: dt \nonumber \\
&\leq \sup_{0 \leq t \leq T_n + s} |S_n(s,t) - S(s,t)|.
\end{align}
Therefore, $$\sup_{0 \leq s \leq T} \int_0^{T_n + s} |Y_n(s,t) - Y(s,t)| \: dt \leq \sup_{0 \leq s \leq T} \sup_{0 \leq t \leq T_n + s} |S_n(s,t) - S(s,t)| \rightarrow_p 0$$ by Lemma \ref{supSnS}.
\end{proof}

Choose fixed times $0 = s_0 < s_1 < \dots < s_m = T$ such that $1/(\log n)^2 \leq s_{i+1} - s_i \leq 2/(\log n)^2$ for $i = 0, 1, \dots, m-1$.  This is clearly possible for sufficiently large $n$, and $m \leq T(\log n)^2$.  
Let $\eps > 0$, and let $(-u_1, y_1), \dots, (-u_k, y_k)$ denote the points of $\Psi$ in the region $[-T, 0] \times [\eps^3, \infty)$.  Note that with probability one, there are only finitely many points of $\Psi$ in this region.  For $i = 0, 1, \dots, m-1$, let $G_i$ be the event that none of the $u_j$ falls in $[s_i, s_{i+1}]$, and let $H_i$ be the event that exactly one of the $u_j$ falls in $[s_i, s_{i+1}]$ and that this point $u_j$ is in $(s_i, s_{i+1})$.  Note that almost surely none of the $u_j$ land on one of the points $s_0, \dots, s_m$, and with probability tending to one as $n \rightarrow \infty$, no two of the points $u_j$ fall in the same interval $[s_i, s_{i+1}]$.  Consequently, we have
\begin{equation}\label{GHprob}
\lim_{n \rightarrow \infty} P \bigg( \bigcap_{i=0}^{m-1} (G_i \cup H_i) \bigg) = 1.
\end{equation}
On $H_i$, let $j(i)$ be the value of $j$ such that $u_j \in (s_i, s_{i+1})$.

\begin{Lemma}\label{YA}
We have $$\limsup_{n \rightarrow \infty} P \bigg( \sup_{0 \leq i \leq m-1} \sup_{s \in [s_i, s_{i+1}]} \bigg| \int_0^{T_n + s} Y_n(s,t) \: dt - \int_0^{T_n + s_{i+1}} Y_n(s_{i+1}, t) \: dt \bigg| {\bf 1}_{G_i} \geq 4 \eps \bigg) \leq 4 T \eps.$$
\end{Lemma}

\begin{proof}
If $s_i \leq s \leq s_{i+1}$ and $t \geq 0$, then $$S_n(s,t) = S_n(s_{i+1}, t + s_{i+1} - s) - S_n(s_{i+1}, s_{i+1} - s).$$  Therefore,
\begin{align}\label{Yssi}
\int_0^{T_n + s} Y_n(s,t) \: dt &= \int_0^{T_n + s} e^{-t} S_n(s_{i+1}, t + s_{i+1} - s) \: dt \nonumber \\
&\hspace{.5in} - \int_0^{T_n + s} e^{-t} S_n(s_{i+1}, s_{i+1} - s) \: dt + \int_0^{T_n + s} \frac{e^{-t} t^2}{2} \: dt \nonumber \\
&= e^{s_{i+1} - s} \int_0^{T_n + s_{i+1}} e^{-t} S_n(s_{i+1}, t) \: dt - e^{s_{i+1} - s} \int_0^{s_{i+1} - s} e^{-t} S_n(s_{i+1}, t) \: dt \nonumber \\
&\hspace{.5in} - \int_0^{T_n + s} e^{-t} S_n(s_{i+1}, s_{i+1} - s) \: dt + \int_0^{T_n + s} \frac{e^{-t} t^2}{2} \: dt \nonumber \\
&= \int_0^{T_n + s_{i+1}} Y_n(s_{i+1}, t) \: dt + (e^{s_{i+1} - s} - 1) \int_0^{T_n + s_{i+1}} e^{-t} S_n(s_{i+1}, t) \: dt \nonumber \\
&\hspace{.5in} - e^{s_{i+1} - s} \int_0^{s_{i+1} - s} e^{-t} S_n(s_{i+1}, t) \: dt - \int_0^{T_n + s} e^{-t} S_n(s_{i+1}, s_{i+1} - s) \: dt \nonumber \\
&\hspace{.5in} - \int_{T_n + s}^{T_n + s_{i+1}} \frac{e^{-t} t^2}{2} \: dt. 
\end{align}
We must bound the last four terms on the right-hand side of (\ref{Yssi}).
Now $e^{s_{i+1} - s_i} - 1 \leq 3/(\log n)^2$ for sufficiently large $n$, so 
\begin{align}\label{ET1}
\bigg| (e^{s_{i+1} - s} - 1) &\int_0^{T_n + s_{i+1}} e^{-t} S_n(s_{i+1}, t) \: dt \bigg| \nonumber \\
&\leq \frac{3}{(\log n)^2} \bigg(1 + \bigg| \int_0^{T_n + s_{i+1}} Y_n(s_{i+1}, t) \: dt \bigg| \bigg) \nonumber \\
&\leq \frac{3}{(\log n)^2} \bigg(1 + \int_0^{T_n + s_{i+1}} |Y_n(s_{i+1},t) - Y(s_{i+1},t)| \: dt \nonumber \\
&\hspace{.5in} + \bigg|\int_{T_n + s_{i+1}}^{\infty} Y(s_{i+1},t) \: dt \bigg| + \bigg| \int_0^{\infty} Y(s_{i+1}, t) \: dt \bigg| \bigg) \nonumber \\
&\leq \frac{3}{(\log n)^2} \bigg(1 + \sup_{0 \leq s \leq T} \int_0^{T_n + s} |Y_n(s,t) - Y(s,t)| \: dt \nonumber \\
&\hspace{.5in}+ \sup_{0 \leq s \leq T} \bigg|\int_{T_n + s}^{\infty} Y(s,t) \: dt \bigg| +  \sup_{0 \leq s \leq T} |L(s)| + \sup_{0 \leq s \leq T} |J(s)| \bigg).
\end{align}
Also, $e^{s_{i+1} - s} \leq 2$ for sufficiently large $n$, which means
\begin{align}\label{ET2}
\bigg| e^{s_{i+1} - s} \int_0^{s_{i+1} - s} e^{-t} S_n(s_{i+1}, t) \: dt \bigg| &\leq 2 (s_{i+1} - s_i) \sup_{0 \leq t \leq s_{i+1} - s_i} |S_n(s_{i+1}, t)| \nonumber \\
&\leq \frac{4}{(\log n)^2} \sup_{0 \leq t \leq s_{i+1} - s_i} |S_n(s_{i+1}, t)|.
\end{align}
Likewise,
\begin{equation}\label{ET3}
\bigg| \int_0^{T_n + s} e^{-t} S_n(s_{i+1}, s_{i+1} - s) \: dt \bigg| \leq \sup_{0 \leq t \leq s_{i+1} - s_i} |S_n(s_{i+1}, t)|.
\end{equation}
Finally,
\begin{equation}\label{ET4}
\bigg|\int_{T_n + s}^{T_n + s_{i+1}} \frac{e^{-t} t^2}{2} \: dt \bigg| \leq \frac{1}{(\log n)^2}
\end{equation}
for sufficiently large $n$.    Combining (\ref{ET1}), (\ref{ET2}), (\ref{ET3}), and (\ref{ET4}) with (\ref{Yssi}), we get that for sufficiently large $n$,
\begin{align}\label{supsup}
&\sup_{0 \leq i \leq m-1} \sup_{s \in [s_i, s_{i+1}]} \bigg| \int_0^{T_n + s} Y_n(s,t) \: dt - \int_0^{T_n + s_{i+1}} Y_n(s_{i+1}, t) \: dt \bigg| \nonumber \\
&\hspace{.1in} \leq \frac{4}{(\log n)^2} + 2 \sup_{0 \leq i \leq m-1} \sup_{0 \leq t \leq s_{i+1} - s_i} |S_n(s_{i+1}, t)| + \frac{3}{(\log n)^2} \bigg( \sup_{0 \leq s \leq T} \int_0^{T_n + s} |Y_n(s,t) - Y(s,t)| \: dt \nonumber \\
&\hspace{1.5in}+  \sup_{0 \leq s \leq T} \bigg|\int_{T_n + s}^{\infty} Y(s,t) \: dt \bigg| + \sup_{0 \leq s \leq T} |L(s)| + \sup_{0 \leq s \leq T} |J(s)| \bigg).
\end{align}
Note that $\sup_{0 \leq s \leq T} |J(s)|$ is almost surely finite by properties of the Poisson process $\Psi$, and $\sup_{0 \leq s \leq T} |L(s)|$ is almost surely finite by Proposition \ref{LOUprop}.  Combining these observations with Lemmas \ref{Ytailint} and \ref{YnYlem}, we see that the probability that the third term on the right-hand side of (\ref{supsup}) is less than $\eps$ tends to zero as $n \rightarrow \infty$.  Clearly $4/(\log n)^2 < \eps$ for sufficiently large $n$.  To control the second term, apply Lemma \ref{Sntheta} with $\delta_n = s_{i+1} - s_i$ and $\theta = \eps^3$ to get $$P \bigg( \sup_{0 \leq t \leq s_{i+1} - s_i} |S_n(s_{i+1}, t)| {\bf 1}_{G_i} > \eps \bigg) \leq P \bigg( \sup_{0 \leq t \leq s_{i+1} - s_i} |S_n(s_{i+1}, t)| > \eps \bigg| G_i \bigg) \leq 4 (s_{i+1} - s_i) \eps,$$ and thus $$P \bigg( 2 \sup_{0 \leq i \leq m-1} \sup_{0 \leq t \leq s_{i+1} - s_i} |S_n(s_{i+1}, t)| {\bf 1}_{G_i} > 2 \eps \bigg) \leq 4 \eps \sum_{i=0}^{m-1} (s_{i+1} - s_i) = 4 T \eps.$$  Combining these bounds with (\ref{supsup}) gives the result.
\end{proof}

\begin{Lemma}\label{YB}
We have
\begin{align}
\lim_{n \rightarrow \infty} P \bigg( \sup_{0 \leq i \leq m-1} \sup_{s \in [s_i, s_{i+1}]} \bigg|  \int_0^{T_n + s} &Y_n(s,t) \: dt - J(u_{j(i)}) {\bf 1}_{\{s \leq u_{j(i)}\}}  \nonumber \\
&- \int_0^{T_n + s_{i+1}} Y_n(s_{i+1}, t) \: dt \bigg| {\bf 1}_{H_i} \geq 5 \eps \bigg) = 0. \nonumber
\end{align}
\end{Lemma}

\begin{proof}
Note that (\ref{Yssi}) still holds in this setting.  Consequently, if $s_i \leq s \leq s_{i+1}$, then on the event $H_i$, we have
\begin{align}\label{qw1}
\bigg| \int_0^{T_n + s} &Y_n(s,t) \: dt - J(u_{j(i)}) {\bf 1}_{\{s \leq u_{j(i)}\}} - \int_0^{T_n + s_{i+1}} Y_n(s_{i+1}, t) \: dt \bigg| \nonumber \\
&\leq \bigg| (e^{s_{i+1} - s} - 1) \int_0^{T_n + s_{i+1}} e^{-t} S_n(s_{i+1}, t) \: dt \bigg| + \bigg| e^{s_{i+1} - s} \int_0^{s_{i+1} - s} e^{-t} S_n(s_{i+1}, t) \: dt \bigg| \nonumber \\
&\hspace{.2in}+ \bigg| \int_0^{T_n + s} e^{-t} S_n(s_{i+1}, s_{i+1} - s) \: dt + J(u_{j(i)}) {\bf 1}_{\{s \leq u_{j(i)}\}} \bigg| + \bigg| \int_{T_n+s}^{T_n + s_{i+1}} \frac{e^{-t} t^2}{2} \: dt \bigg|.
\end{align}
The bounds (\ref{ET1}), (\ref{ET2}), and (\ref{ET4}) also hold.  In place of (\ref{ET3}), observe that if $s_i \leq s \leq s_{i+1}$, then on the event $H_i$, 
\begin{align}\label{qw2}
\bigg| &\int_0^{T_n + s} e^{-t} S_n(s_{i+1}, s_{i+1} - s) \: dt + J(u_{j(i)}){\bf 1}_{\{s \leq u_{j(i)}\}} \bigg| \nonumber \\
&\hspace{.2in}= \bigg| S_n(s_{i+1}, s_{i+1} - s) \bigg(1 - \int_{T_n+s}^{\infty} e^{-t} \: dt \bigg) + J(u_{j(i)}) {\bf 1}_{\{s \leq u_{j(i)}\}} \bigg| \nonumber \\
&\hspace{.2in}\leq \big| S_n(s_{i+1}, s_{i+1} - s) + J(u_{j(i)}){\bf 1}_{\{s \leq u_{j(i)}\}} \big| + \big| S_n(s_{i+1}, s_{i+1} - s) e^{-(T_n + s)} \big| \nonumber \\
&\hspace{.2in}\leq \sup_{0 \leq t \leq s_{i+1} - s_i} \big|S_n(s_{i+1}, t) + J(u_{j(i)}) {\bf 1}_{\{t \geq s_{i+1} - u_{j(i)}\}} \big| + \frac{1}{(\log n)^2} \sup_{0 \leq t \leq s_{i+1} - s_i} |S_n(s_{i+1}, t)|. 
\end{align}
Also, on the event $H_i$,
\begin{equation}\label{qw3}
\sup_{0 \leq t \leq s_{i+1} - s_i} |S_n(s_{i+1}, t)| \leq |J(u_{j(i)})| + \sup_{0 \leq t \leq s_{i+1} - s_i} \big|S_n(s_{i+1}, t) - J(u_{j(i)}) {\bf 1}_{\{t \geq s_{i+1} - u_{j(i)}\}} \big|.
\end{equation}
Thus, on the event $H_i$, by (\ref{qw1}), (\ref{qw2}), (\ref{qw3}), (\ref{ET1}), (\ref{ET2}), and (\ref{ET4}), we have
\begin{align}\label{supsup2}
&\sup_{0 \leq i \leq m-1} \sup_{s \in [s_i, s_{i+1}]} \bigg|  \int_0^{T_n + s} Y_n(s,t) \: dt - J(u_{j(i)}) {\bf 1}_{\{s \leq u_{j(i)}\}} - \int_0^{T_n + s_{i+1}} Y_n(s_{i+1}, t) \: dt \bigg| \nonumber \\
&\hspace{.3in} \leq \frac{4}{(\log n)^2} + \frac{5}{(\log n)^2} \sup_{0 \leq s \leq T} |J(s)| \nonumber \\
&\hspace{.5in} + \bigg(1 + \frac{5}{(\log n)^2} \bigg) \sup_{0 \leq i \leq m-1} \sup_{0 \leq t \leq s_{i+1} - s_i} \big|S_n(s_{i+1}, t) + J(u_{j(i)}) {\bf 1}_{\{t \geq s_{i+1} - u_{j(i)}\}} \big| \nonumber \\
&\hspace{.5in} + \frac{3}{(\log n)^2} \bigg( \sup_{0 \leq s \leq T} \int_0^{T_n + s} |Y_n(s,t) - Y(s,t)| \: dt \nonumber \\
&\hspace{1in}+  \sup_{0 \leq s \leq T} \bigg|\int_{T_n + s}^{\infty} Y(s,t) \: dt \bigg| + \sup_{0 \leq s \leq T} |L(s)| +  \sup_{0 \leq s \leq T} |J(s)| \bigg). 
\end{align}
By Lemmas \ref{Ytailint} and \ref{YnYlem}, and the fact that $\sup_{0 \leq s \leq T} |L(s)|$ and $\sup_{0 \leq s \leq T} |J(s)|$ are finite almost surely, the probability that the first, second, or fourth term on the right-hand side of (\ref{supsup2}) is greater than $\eps$ tends to zero as $n \rightarrow \infty$.  To bound the third term, we apply Lemma \ref{Sntheta}.  Observe that the conditional distribution of $$\sup_{0 \leq t \leq s_{i+1} - s_i} \big|S_n(s_{i+1}, t) + J(u_{j(i)}) {\bf 1}_{\{t \geq s_{i+1} - u_{j(i)}\}}\big|$$ given $H_i$ is the same as the conditional distribution of $$\sup_{0 \leq t \leq \delta_n} |S_n(s,t)|$$ given $A(\theta)$ in Lemma \ref{Sntheta} if we take $\delta_n = s_{i+1} - s_i$ and $\theta = \eps^3$.  Because
\begin{equation}\label{PBi}
P(H_i) \leq (s_{i+1} - s_i) \int_{\eps^3}^{\infty} x^{-2} \: dx \leq \frac{2}{\eps^3 (\log n)^2},
\end{equation}
it follows from Lemma \ref{Sntheta} that
\begin{align}
P \bigg( 2 \sup_{0 \leq i \leq m-1} \sup_{0 \leq t \leq s_{i+1} - s_i} \big|S_n(s_{i+1}, t&) + J(u_{j(i)}) {\bf 1}_{\{t \geq s_{i+1} - u_{j(i)}\}} \big| {\bf 1}_{H_i} > 2 \eps \bigg) \nonumber \\
&\leq \sum_{i=0}^{m-1} P(H_i) \cdot 4 (s_{i+1} - s_i) \eps \leq \frac{8T}{\eps^2 (\log n)^2}, \nonumber
\end{align}
which tends to zero as $n \rightarrow \infty$.  The result follows.
\end{proof}

\begin{Lemma}
If $s \geq 0$ and $h \geq 0$ with $0 \leq s+h \leq T$, then 
\begin{equation}\label{X1}
\int_0^{T_n+s+h} X_n(s+h, t) \: dt \leq \int_0^{T_n+s} X_n(s,t) \: dt + h \log n
\end{equation}
and for sufficiently large $n$,
\begin{align}\label{X2}
\int_0^{T_n+s+h} &X_n(s+h, t) \: dt \nonumber \\
&\geq \int_0^{T_n+s} X_n(s,t) \: dt - \frac{2h}{\log n} - \frac{(\log n)(T_n + s)}{n} \bigg(n - N_n \bigg( \frac{s+h}{\log n}, \frac{h}{\log n} \bigg) + 1 \bigg).
\end{align}
\end{Lemma}

\begin{proof}
We have
\begin{align}\label{Xnsh}
\int_0^{T_n + s + h} X_n(s+h, t) \: dt &= \frac{\log n}{n} \int_0^{T_n + s + h} {\tilde N}_n \bigg( \frac{s+h}{\log n}, \frac{t}{\log n} \bigg) - n e^{-t} - \frac{nt e^{-t} \log \log n}{\log n} \: dt \nonumber \\
&= \frac{\log n}{n} \int_0^{T_n + s} {\tilde N}_n \bigg( \frac{s}{\log n}, \frac{t}{\log n} \bigg) - n e^{-t} - \frac{nt e^{-t} \log \log n}{\log n} \: dt \nonumber \\
&\hspace{.5in}- \frac{\log n}{n} \int_{T_n+s}^{T_n+s+h} \bigg( ne^{-t} + \frac{nte^{-t} \log \log n}{\log n} \bigg) \: dt \nonumber \\
&\hspace{.5in}+ \frac{\log n}{n} \int_0^{T_n+s+h} {\tilde N}_n \bigg( \frac{s+h}{\log n}, \frac{t}{\log n} \bigg) \: dt \nonumber \\
&\hspace{.5in}- \frac{\log n}{n} \int_0^{T_n+s} {\tilde N}_n \bigg( \frac{s}{\log n}, \frac{t}{\log n} \bigg) \: dt \nonumber \\
&= \int_0^{T_n+s} X_n(s,t) \: dt - \frac{\log n}{n} \int_{T_n+s}^{T_n+s+h} \bigg( ne^{-t} + \frac{nte^{-t} \log \log n}{\log n} \bigg) \: dt \nonumber \\
&\hspace{.5in}+ \frac{\log n}{n} \int_0^h {\tilde N}_n \bigg( \frac{s+h}{\log n}, \frac{t}{\log n} \bigg) \: dt \nonumber \\
&\hspace{.5in}+ \frac{\log n}{n} \int_0^{T_n+s} {\tilde N}_n \bigg( \frac{s+h}{\log n}, \frac{t+h}{\log n} \bigg) - {\tilde N}_n \bigg( \frac{s}{\log n}, \frac{t}{\log n} \bigg) \: dt.
\end{align}
Because $N_n(s,t) \leq n$ for all $s,t \geq 0$, the integral in the third term on the right-hand side of (\ref{Xnsh}) is bounded by $hn$, so the term is bounded by $h \log n$.  Also, for all $s,t,h \geq 0$, we have $N(s+h, t+h) \leq N(s,t)$ because the number of individuals in the population immediately before time $(s+h)-(t+h) = s-t$ with descendants alive in the population at time $s+h$ is less than or equal to the number of individuals in the population immediately before time $s-t$ with descendants alive in the population at time $s$.  Therefore, the fourth term in the right-hand side of (\ref{Xnsh}) is nonpositive.  Equation (\ref{X1}) follows.

Next, we claim that for all $s,t,h \geq 0$, we have $N_n(s,t) - N_n(s+h, t+h) \leq n - N_n(s+h, h)$.  To see this, note that $N_n(s,t) - N_n(s+h, t+h)$ is the number of individuals in the population immediately before time $s-t$ with descendants alive in the population at time $s$ but not at time $s+h$.  This is at most the number of individuals in the population at time $s$ that do not have descendants alive in the population at time $s+h$, which is at most $n - N_n(s+h, h)$.  Therefore, ${\tilde N}_n(s,t) - {\tilde N}_n(s+h, t+h) \leq n - N_n(s+h,h) + 1$.  The result (\ref{X2}) follows from this observation and (\ref{Xnsh}) because
$$\frac{\log n}{n} \int_{T_n+s}^{T_n+s+h} \bigg( ne^{-t} + \frac{nte^{-t} \log \log n}{\log n} \bigg) \: dt \leq (h \log n) \sup_{t \geq T_n} \bigg( e^{-t} + \frac{te^{-t} \log \log n}{\log n} \bigg) \leq \frac{2h}{\log n}$$ for sufficiently large $n$.
\end{proof}

\begin{proof}[Proof of Proposition \ref{LWcouple}]
By Lemmas \ref{Ytailint}, \ref{Xtailint}, and \ref{YnYlem}, it suffices to show that
\begin{equation}\label{mainsts}
\sup_{0 \leq s \leq T} \bigg| \int_0^{T_n + s} X_n(s,t) \: dt + J(s) - \int_0^{T_n + s} Y_n(s,t) \: dt \bigg| \rightarrow_p 0.
\end{equation}
By Lemmas \ref{Anlem}, \ref{Bnlem}, and \ref{mainXYprob2}, we have
\begin{align}
\limsup_{n \rightarrow \infty} P \bigg( &\int_0^{T_n + s_i} |X_n(s_i, t) - Y_n(s_i, t)| \: dt \geq \frac{5(T_n + s_i)}{(\log n)^{1/16}} \mbox{ for some }i = 0, 1, \dots, m \bigg) \nonumber \\
&\leq \limsup_{n \rightarrow \infty} \bigg(P(A_n^c) + P(B_n^c) \nonumber \\
&\hspace{.3in}+ (1 + T(\log n)^2) \bigg(\frac{(T_n + T)(\log n)^{3/2}}{n^{1/2} \eps_n} + 4(\log n)^2 (T_n + T) \eps_n \bigg) \bigg) = 0. \nonumber
\end{align}
Therefore, for all $\eps > 0$, we have
\begin{equation}\label{XYsi}
\lim_{n \rightarrow \infty} P \bigg( \sup_{0 \leq i \leq m} \bigg| \int_0^{T_n + s_i} X_n(s_i, t) \: dt - \int_0^{T_n + s_i} Y_n(s_i, t) \: dt \bigg| \geq \eps \bigg) = 0.
\end{equation}

We first consider the case in which $G_i$ occurs.  If $s \in [s_i, s_{i+1}]$, then by (\ref{X1}),
\begin{equation}\label{895}
\int_0^{T_n+s} X_n(s,t) \leq \int_0^{T_n+s_i} X_n(s_i, t) \: dt + (s - s_i) \log n, 
\end{equation}
and therefore
\begin{align}\label{Xnupper}
\int_0^{T_n+s} X_n(s,t) &\leq \int_0^{T_n+s} Y_n(s, t) \: dt + \bigg| \int_0^{T_n+s_i} X_n(s_i, t) \: dt - \int_0^{T_n+s_i} Y_n(s_i, t) \: dt \bigg| \nonumber \\
&\hspace{.3in} + \bigg| \int_0^{T_n+s_i} Y_n(s_i, t) \: dt - \int_0^{T_n+s_{i+1}} Y_n(s_{i+1}, t) \: dt \bigg| \nonumber \\
&\hspace{.3in} + \bigg| \int_0^{T_n+s_{i+1}} Y_n(s_{i+1}, t) \: dt - \int_0^{T_n+s} Y_n(s, t) \: dt \bigg| + (s - s_i) \log n. 
\end{align}
Likewise,
\begin{equation}\label{905}
\int_0^{T_n+s} X_n(s,t) \geq \int_0^{T_n+s_{i+1}} X_n(s_{i+1}, t) \: dt - (s_{i+1} - s) \log n,
\end{equation}
and so
\begin{align}\label{Xnlower}
\int_0^{T_n+s} X_n(s,t) &\geq \int_0^{T_n+s} Y_n(s, t) \: dt - \bigg| \int_0^{T_n+s_{i+1}} X_n(s_{i+1}, t) \: dt - \int_0^{T_n+s_{i+1}} Y_n(s_{i+1}, t) \: dt \bigg| \nonumber \\
&\hspace{.3in} - \bigg| \int_0^{T_n+s_{i+1}} Y_n(s_{i+1}, t) \: dt - \int_0^{T_n+s} Y_n(s, t) \: dt \bigg| - (s_{i+1} - s) \log n.
\end{align}
Since $s_{i+1} - s_i \leq 2/(\log n)^2$ for $i = 0, 1, \dots, m-1$, it follows that 
\begin{align}
\sup_{0 \leq i \leq m-1} &\sup_{s \in [s_i, s_{i+1}]} \bigg| \int_0^{T+s} X_n(s, t) \: dt - \int_0^{T+s} Y_n(s,t) \: dt \bigg| \nonumber \\
&\leq \frac{2}{\log n} + \sup_{0 \leq i \leq m}  \bigg| \int_0^{T_n + s_i} X_n(s_i, t) \: dt - \int_0^{T_n + s_i} Y_n(s_i, t) \: dt \bigg| \nonumber \\
&\hspace{.3in}+ 2 \sup_{0 \leq i \leq m-1} \sup_{s \in [s_i, s_{i+1}]} \bigg| \int_0^{T_n + s} Y_n(s,t) \: dt - \int_0^{T_n + s_{i+1}} Y_n(s_{i+1}, t) \: dt \bigg|. \nonumber
\end{align}
Clearly $2/\log n < \eps$ for sufficiently large $n$.  Therefore, by (\ref{XYsi}) and Lemma \ref{YA},
\begin{equation}\label{finA}
\limsup_{n \rightarrow \infty} P \bigg( \sup_{0 \leq i \leq m-1} \sup_{s \in [s_i, s_{i+1}]} \bigg| \int_0^{T+s} X_n(s, t) \: dt - \int_0^{T+s} Y_n(s,t) \: dt \bigg| {\bf 1}_{G_i} > 10 \eps \bigg) \leq 4T\eps.
\end{equation}

We next consider the case in which $H_i$ occurs.  Recall that on $H_i$, we have $u_{j(i)} \in (s_i, s_{i+1})$.  Let $$R_i = \lim_{s \uparrow u_{j(i)}} N_n \bigg( \frac{s}{\log n}, \frac{s - s_i}{\log n} \bigg).$$  That is, $R_i$ is the number of individuals in the population immediately before time $s_i/\log n$ with descendants in the population at least until immediately before time $u_{j(i)}/\log n$.  Note that $N_n(s/\log n, (s - s_i)/\log n) \geq R_i$ for all $s \in [s_i, u_{j(i)})$.  Therefore, by (\ref{X2}), for $s \in [s_i, u_{j(i)})$ we have
$$\int_0^{T_n + s} X_n(s,t) \: dt \geq \int_0^{T_n + s_i} X_n(s_i,t) \: dt - \frac{2(s - s_i)}{\log n} - \frac{(\log n)(T_n + T)(n - R_i + 1)}{n}.$$ Combining this result with (\ref{895}) and the fact that $s_{i+1} - s_i \leq 2/\log n$ gives
$$\bigg| \int_0^{T_n+s} X_n(s,t) \: dt - \int_0^{T_n + s_i} X_n(s_i,t) \: dt \bigg| \leq \frac{2}{\log n} + \frac{4}{(\log n)^3} + \frac{(\log n)(T_n + T)(n - R_i + 1)}{n}.$$  It follows that if $s \in [s_i, u_{j(i)})$, then
\begin{align}\label{XYB1}
\bigg| \int_0^{T_n+s} &X_n(s,t) \: dt - \int_0^{T_n+s} Y_n(s,t) \: dt \bigg| \nonumber \\
&\leq \frac{2}{\log n} + \frac{4}{(\log n)^3} + \frac{(\log n)(T_n + T)(n - R_i + 1)}{n} \nonumber \\
&\hspace{.3in}+ \bigg| \int_0^{T_n+s_i} X_n(s_i, t) \: dt - \int_0^{T_n+s_i} Y_n(s_i, t) \: dt \bigg| \nonumber \\
&\hspace{.3in}+ \bigg| \int_0^{T_n+s_i} Y_n(s_i, t) \: dt - J(u_{j(i)}) - \int_0^{T_n + s_{i+1}} Y_n(s_{i+1}, t) \: dt \bigg| \nonumber \\
&\hspace{.3in}+  \bigg| \int_0^{T_n+s} Y_n(s, t) \: dt - J(u_{j(i)}) - \int_0^{T_n + s_{i+1}} Y_n(s_{i+1}, t) \: dt \bigg|. 
\end{align}
Likewise, let $$S_i = N_n \bigg( \frac{s_{i+1}}{\log n}, \frac{s_{i+1} - u_{j(i)}}{\log n}- \bigg),$$ and note that $N_n(s_{i+1}/\log n, (s_{i+1} - s)/\log n) \geq S_i$ for all $s \in (u_{j(i)}, s_{i+1}]$.  Therefore, if $s \in (u_{j(i)}, s_{i+1}]$, then (\ref{X2}) gives
\begin{equation}\label{last}
\int_0^{T_n+s} X_n(s,t) \: dt \leq \int_0^{T_n + s_{i+1}} X_n(s_{i+1},t) \: dt + \frac{2(s_{i+1} - s)}{\log n} + \frac{(\log n)(T_n + T)(n - S_i + 1)}{n}.
\end{equation}
Because $X_n(u_{j(i)}, t) = \lim_{s \downarrow u_{j(i)}} X_n(s,t)$ for Lebesgue almost all $t > 0$, equation (\ref{last}) holds for all $s \in [u_{j(i)}, s_{i+1}]$.  Combining this result with (\ref{905}) gives
$$\bigg| \int_0^{T_n+s} X_n(s,t) \: dt - \int_0^{T_n + s_{i+1}} X_n(s_{i+1},t) \: dt \bigg| \leq \frac{2}{\log n} + \frac{4}{(\log n)^3} + \frac{(\log n)(T_n + T)(n - S_i + 1)}{n}.$$  Therefore, if $s \in [u_{j(i)}, s_{i+1}]$, then
\begin{align}\label{XYB2}
\bigg| \int_0^{T_n+s} &X_n(s,t) \: dt + J(s) - \int_0^{T_n+s} Y_n(s,t) \: dt \bigg| \nonumber \\
&\leq \frac{2}{\log n} + \frac{4}{(\log n)^3} + \frac{(\log n)(T_n + T)(n - S_i + 1)}{n} \nonumber \\
&\hspace{.3in}+ \bigg| \int_0^{T_n+s_{i+1}} X_n(s_{i+1}, t) \: dt - \int_0^{T_n+s_{i+1}} Y_n(s_{i+1}, t) \: dt \bigg| \nonumber \\
&\hspace{.3in}+ \bigg| \int_0^{T_n+s} Y_n(s,t) - J(s) - \int_0^{T_n + s_{i+1}} Y_n(s_{i+1}, t) \: dt \bigg|.
\end{align}
By (\ref{XYB1}) and (\ref{XYB2}), and the fact that $J(s) = 0$ if $s \in [s_i, u_{j(i)}) \cup (u_{j(i)}, s_{i+1}]$, 
\begin{align}\label{supsup3}
&\sup_{0 \leq i \leq m-1} \sup_{s \in [s_i, s_{i+1}]} \bigg| \int_0^{T+s} X_n(s,t) \: dt + J(s) - \int_0^{T_n+s} Y_n(s,t) \: dt \bigg| {\bf 1}_{H_i} \nonumber \\
&\hspace{.2in}\leq \frac{2}{\log n} + \frac{4}{(\log n)^3} + \sup_{0 \leq i \leq m-1} \frac{(\log n)(T_n + T)(2n - R_i - S_i + 2) {\bf 1}_{H_i}}{n} \nonumber \\
&\hspace{.5in}+ \sup_{0 \leq i \leq m} \bigg| \int_0^{T_n+s_i} X_n(s_i, t) \: dt - \int_0^{T_n+s_i} Y_n(s_i, t) \: dt \bigg| \nonumber \\
&\hspace{.5in}+ 2 \sup_{0 \leq i \leq m-1} \sup_{s \in [s_i, s_{i+1}]} \bigg| \int_0^{T_n+s} Y_n(s,t) - J(u_{j(i)}) {\bf 1}_{\{s \leq u_{j(i)}\}} - \int_0^{T_n+s_{i+1}} Y_n(s_{i+1}, t) \: dt \bigg| {\bf 1}_{H_i}.
\end{align}

We claim that conditional distributions of $n-R_i$ and $n-S_i$ given $H_i$ are each stochastically dominated by the distribution of $n - N_n(0, 2/(\log n)^3)$, which is the number of blocks lost by time $2/(\log n)^3$ in a Bolthausen-Sznitman coalescent started with $n$ blocks.  To see this, note that $n - S_i$ is the number of blocks lost by time $(s_{i+1} - u_{j(i)})/\log n \leq 2/(\log n)^3$ in a Bolthausen-Sznitman coalescent.  Likewise $n - R_i$ is the number of blocks lost by time $(u_{j(i)} - s_i)/(\log n) \leq 2/(\log n)^3$ in a Bolthausen-Sznitman coalescent started with $n$ lineages, if we disallow the instantaneous merger caused by the birth event at time $u_{j(i)}$.  Because $H_i$ requires that exactly one of the $u_k$ falls in $(s_i, s_{i+1})$, the effect of conditioning on $H_i$ is the same as suppressing all mergers in which each lineage participates with probability greater than $\eps^3$.  This can only reduce the number of blocks lost.  Thus, we have $$E[n - R_i|H_i] \leq \frac{2 \gamma_n}{(\log n)^3} \leq \frac{2n}{(\log n)^2}$$ by (\ref{gamnbd}).  By the same argument, $E[n - S_i|H_i] \leq 2n/(\log n)^2$.  Thus, using (\ref{PBi}) and the fact that $m \leq T(\log n)^2$,
\begin{align}
E \bigg[ &\sup_{0 \leq i \leq m-1} \frac{(\log n)(T_n+T)(2n - R_i - S_i + 2){\bf 1}_{H_i}}{n} \bigg] \nonumber \\
&\leq \frac{(\log n)(T_n+T)}{n} \sum_{i=0}^{m-1} P(H_i) \bigg(\frac{4n}{(\log n)^2} + 2 \bigg) \leq \frac{2(\log n)(T_n+T)m}{\eps^3 n (\log n)^2} \bigg( \frac{4n}{(\log n)^2} + 2 \bigg) \rightarrow 0 \nonumber
\end{align}
as $n \rightarrow \infty$.  Hence, by Markov's Inequality,
$$\frac{2}{\log n} + \frac{4}{(\log n)^3} + \sup_{0 \leq i \leq m-1} \frac{(\log n)(T_n + T)(2n - R_i - S_i + 2) {\bf 1}_{H_i}}{n} \rightarrow_p 0.$$  Combining this result with (\ref{supsup3}), (\ref{XYsi}), and Lemma \ref{YB}, we get
\begin{equation}\label{finB}
\limsup_{n \rightarrow \infty} P \bigg( \sup_{0 \leq i \leq m-1} \sup_{s \in [s_i, s_{i+1}]} \bigg| \int_0^{T+s} X_n(s,t) \: dt + J(s) - \int_0^{T+s} Y_n(s,t) \: dt \bigg| {\bf 1}_{H_i} > 12 \eps \bigg) = 0.
\end{equation}
The result (\ref{mainsts}) follows from (\ref{finA}), (\ref{finB}), and the fact that with probability tending to $1$ as $n \rightarrow \infty$, we have $G_i \cup H_i$ for $i = 0, 1, \dots, m-1$ by (\ref{GHprob}).
\end{proof}

\end{document}